\documentclass[12pt]{amsart}

\usepackage{amsmath, amsthm, amssymb}
\usepackage[margin=1in]{geometry}
\usepackage{tikz}
\usepackage{mathtools}
\usepackage{multirow}

\allowdisplaybreaks

\numberwithin{equation}{section}

\newcommand{\Z}{\mathbb{Z}}
\newcommand{\N}{\mathbb{N}}

\newcommand{\R}{\mathbb{R}}
\newcommand{\F}{\mathcal{F}}
\newcommand{\Sh}{\mathcal{S}}
\newcommand{\Mu}{\boldsymbol{\mu}}
\newcommand{\Nu}{\boldsymbol{\nu}}
\newcommand{\Y}{\boldsymbol{y}}

\newcommand{\bl}{\boldsymbol{\ell}}

\newcommand{\supp}{\mathop{\mathrm{supp}}}
\newcommand{\pa}{\partial}
\newcommand{\vphi}{\varphi}
\newcommand{\Op}{\mathop{\mathrm{Op}}}

\newcommand{\la}{\langle}
\newcommand{\ra}{\rangle}

\newcommand{\Xxi}{\boldsymbol{\xi}}

\theoremstyle{plain}
\newtheorem{thm}{Theorem}[section]
\newtheorem{prop}[thm]{Proposition}
\newtheorem{lem}[thm]{Lemma}

\newtheorem*{thmA}{Theorem A}
\newtheorem*{thmB}{Theorem B}
\newtheorem*{thmC}{Theorem C}
\newtheorem*{thmD}{Theorem D}

\theoremstyle{definition}

\newtheorem{rem}[thm]{Remark}

\begin{document}

\title[Multilinear pseudo-differential operators]
{Boundedness of multilinear pseudo-differential operators with $S_{0,0}$ class symbols on Besov spaces}

\author[N. Shida]{Naoto Shida}

\date{\today}

\address[N. Shida]
{Graduate School of Mathematics, Nagoya University, Furocho, Chikusaku, Nagoya,
Aichi, 464-8602, Japan}

\email[N. Shida]{naoto.shida.c3@math.nagoya-u.ac.jp}


\thanks{This work was supported by Grant-in-Aid for JSPS KAKENHI Fellows, 
Grant Numbers 23KJ1053.}


\begin{abstract}
We consider multilinear pseudo-differential operators with symbols in the multilinear H\"ormander class $S_{0, 0}$.
The aim of this paper is to discuss the boundedness of these operators in the settings of Besov spaces. 

\end{abstract}

\maketitle

\section{Introduction}
For a bounded function $\sigma = \sigma(x, \xi_1, \dots, \xi_N)$ on $(\R^n)^{N+1}$, the ($N$-fold) multilinear  pseudo-differential operator is defined by
\[
T_\sigma(f_1, \dots, f_N)(x)
=
\frac{1}{(2\pi)^{Nn}}
\int_{(\R^n)^N}
e^{i x \cdot (\xi_1 + \dots + \xi_N)}
\sigma(x, \xi_1, \dots, \xi_N)
\prod_{j=1}^N
\widehat{f}_j(\xi_j)
\,
d\xi_1
\dots
d\xi_N,
\]
where $x \in \R^n$, $f_j \in \Sh(\R^n)$, $j=1, \dots, N$, and $\widehat{f}_j$ denotes the Fourier transform of $f_j$.

For $m \in \R$, 
the symbol class $S^m_{0, 0}(\R^n, N)$ denotes the set of all $\sigma = \sigma(x, \xi_1, \dots, \xi_N) \in C^{\infty}((\R^n)^{N+1})$ satisfying
\[
|
\pa^\alpha_x
\pa^{\beta_1}_{\xi_1}
\dots
\pa^{\beta_N}_{\xi_N}
\sigma(x, \xi_1, \dots, \xi_N)
|
\le
C_{\alpha, \beta_1, \dots, \beta_N}
(1+|\xi_1|+ \dots + |\xi_N|)^{m}
\]
for all multi-indices $\alpha, \beta_1, \dots, \beta_N \in \N^n_0 = \{0, 1, 2,  \dots\}^n$.

The subject of this paper is to study the boundedness of multilinear pseudo-differential operators. 
We will use the following notations.
Let $X_1, \dots, X_N$, and $Y$ be function spaces on $\R^n$ equipped with quasi-norms  
$\|\cdot\|_{X_j}$ and $\|\cdot\|_{Y}$, respectively.
We say that $T_{\sigma}$ is bounded from $X_1 \times \dots \times X_N$ to $Y$
if there exists a positive constant $C$ such that the inequality 
\begin{equation}\label{bdd-dfn}
\|T_{\sigma}(f_1, \dots, f_N)\|_{Y}
\le
C
\prod_{j=1}^N
\|f_j\|_{X_j}
\end{equation}
holds for all $f_j \in \Sh \cap X_j$, $j=1, \dots, N$. 
The smallest constant $C$ of \eqref{bdd-dfn} is defined by 
$\|T_{\sigma}\|_{X_1 \times \dots \times X_N \to Y}$. 
If $T_{\sigma}$ is bounded from $X_1 \times \dots \times X_N$ to $Y$ 
for all $\sigma \in S^m_{0,0}(\R^n, N)$, then we write
\[
\Op(S^m_{0,0}(\R^n, N)) \subset B(X_1 \times \dots \times X_N \to Y).
\]

In the linear case, the class $S^m_{0,0}(\R^n, 1)$
 is well known as the H\"ormander symbol class of $S_{0,0}$, 
and the boundedness of linear pseudo-differential operators with symbols in this class
was well studied.
More precisely, the following is widely known
(for the definition of the function spaces 
$h^p$ and $bmo$,
see Section \ref{preliminaries}).
\begin{thmA}[\cite{CV, CM-Asterisque, Miyachi-MathNachr, PS, KMT-JFA}]
Let 
$0< p \le \infty$.
Then the boundedness
\[
\Op(
S^{m}_{0, 0}
)
\subset
B(h^{p} \to h^{p}).
\]
holds if and only if
\[
m 
\le
-n\left|\frac{1}{p} - \frac{1}{2} \right|,
\]
where $h^{p}$ 
should be replaced by $bmo$ if $p = \infty$ .
\end{thmA}

In Theorem A,
the ``if'' part was proved by 
Calder\'on--Vaillancourt \cite{CV} for the case $p=2$,
Coifman--Meyer \cite{CM-Asterisque} for the case $1 < p < \infty$,
and 
Miyachi \cite{Miyachi-MathNachr} and P\"aiv\"arinta--Somersalo \cite{PS} for the case $0 < p \le \infty$.
The proof of ``only if'' part can be found in \cite{KMT-JFA}.

As a generalization of Theorem A, 
the boundedness of linear pseudo-differential operators of $S_{0,0}$
was studied in the settings of  Besov spaces $B^s_{p, q}$
(for the definition of Besov spaces, see Section \ref{preliminaries}).
\begin{thmB}[\cite{Marschall, Sugimoto, Park}]
Let $0< p \le \infty$, $0< q \le \infty$, and $s, t \in \R$.
Let
\[
m
= 
-n\left|\frac{1}{p} - \frac{1}{2} \right|
+s-t.
\]
Then
\[
\Op(S^{m}_{0, 0})
\subset
B(B^s_{p, q} \to B^{t}_{p, q}).
\]
\end{thmB}
The above result was given by 
Marschall \cite{Marschall} for the case $p=q=\infty$, 
and by Sugimoto \cite{Sugimoto} for the case $1 \le p, q \le \infty$.
Recently, Theorem B was proved by Park \cite{Park}.

In the multilinear settings, 
the class $S^m_{0,0}(\R^n, N)$ was first studied by B\'enyi-Torres \cite{BT-MRL}
for the case $N=2$, that is, the bilinear case.
The authors proved that, 
for $1 \le p, p_1, p_2 < \infty$ satisfying $1/p = 1/p_1 + 1/p_2$, 
there exist $x$-independent symbols in $S^0_{0,0}(\R^n, 2)$ such that 
the corresponding  bilinear operators are not bounded from $L^{p_1} \times L^{p_2}$ to $L^p$. 
In particular, they pointed out that the class $S^0_{0,0}(\R^n, 2)$ does not assure 
the $L^2 \times L^2 \to L^1$ boundedness in contrast to the Calder\'on-Vaillancourt theorem for linear pseudo-differential operators.
Then, the number $m$ which assures the $L^{p_1} \times \dots \times L^{p_N} \to L^p$ boundedness of these operators 
was investigated by
B\'enyi-Bernicot-Maldonado-Naibo-Torres \cite{BBMNT}
and
Michalowski-Rule-Staubach \cite{MRS},
and after that,
a complete description of $m$ was given by Miyachi-Tomita \cite{MT-IUMJ}  for $N = 2$ 
(for the case $1/p = 1/p_1+1/p_2$)
and Kato-Miyachi-Tomita\cite{KMT-JFA} for $N \ge 2$.

\begin{thmC}[\cite{MT-IUMJ, KMT-JFA}]
Let $N \ge 2$, $0< p, p_1, \dots, p_N \le \infty$, $1/p \le 1/p_1 + \dots + 1/p_N$ 
and $m \in \R$. 
Then the boundedness
\[
\Op(S^m_{0, 0}(\R^n, N))
\subset
B(h^{p_1} \times \dots \times h^{p_N} \to h^p)
\]
holds if and only if
\[
m 
\le
\min
\left\{
\frac{n}{p},  
\frac{n}{2}
\right\}
-
\sum_{j=1}^N
\max
\left\{
\frac{n}{p_j}, 
\frac{n}{2}
\right\}.
\]
where $h^{p_j}$ can be replaced by $bmo$ if $p_j= \infty$ for some $j=1, \dots, N$.
\end{thmC}


In the scale of Besov spaces, some partial boundedness results for mulitilinear pseudo-differential operators with symbols in $S^m_{0,0}(\R^n, N)$ were given
for the bilinear case.
In Hamada-Shida-Tomita \cite{HST}, it was proved that 
all bilinear pseudo-differential operators with symbols in $S^{-n/2}_{0,0}(\R^n, 2)$
is bounded from $L^2 \times L^2$ to $B^0_{p, q}$ if and only if $1 \le p \le 2$ and $1 \le q \le \infty$.
Since $B^0_{1, 1} \hookrightarrow h^1$, this result improves the $L^2 \times L^2 \to h^1$ boundedness 
given by \cite{MT-IUMJ}.
In \cite{Shida-PAMS}, the following is given.
\begin{thmD}[\cite{Shida-PAMS}]
Let $1 \le p \le 2 \le p_1, p_2 \le \infty$ be such that $1/p \le 1/p_1 + 1/p_2$,
and let $0 < q_1, q_2, q \le \infty$ be such that $1/q \le 1/q_1 + 1/q_2$, 
and let $s_1, s_2, s \in \R$ be such that $s_1+s_2=s$.
If $s_1$, $s_2$ and $s$ satisfy
\begin{equation} \label{condi-PAMS}
s_1 < \frac{n}{2},
\quad
s_2 < \frac{n}{2},
\quad
\text{and}
\quad
s > - \frac{n}{2},
\end{equation}
then
\begin{align*}
\Op(S^{-n/2}_{0, 0}(\R^n, 2))
\subset
B(B^{s_1}_{p_1, q_1} \times B^{s_2}_{p_2, q_2} \to B^{s}_{p, q}).
\end{align*}
\end{thmD}
\noindent
In \cite{Shida-PAMS},
the sharpness of the condition \eqref{condi-PAMS} is also considered. 

The purpose of this paper is  
to extend the partial results on the bilinear case stated in Theorem D to the multilinear case in the full range $0< p, p_1, \dots, p_N \le \infty$.
Our main result reads as follows.

\begin{thm}\label{main1}
Let $N \ge 2$, $0< p, p_1, \dots, p_N \le \infty$, $1/p \le 1/p_1 + \dots + 1/p_N$,
$0 < q, q_1, \dots, q_N \le \infty$, $1/q \le 1/q_1 + \dots + 1/q_N$,
and $s, s_1, \dots, s_N \in \R$.
Let 
\begin{equation}\label{criticalorder}
m
=
\min
\left\{
\frac{n}{p},  
\frac{n}{2}
\right\}
-
\sum_{j=1}^N
\max
\left\{
\frac{n}{p_j}, 
\frac{n}{2}
\right\}
+
\sum_{j=1}^N
s_j
-s.
\end{equation}
If $s_1$, \dots, $s_N$ and $s$ satisfy
\begin{align}\label{sassum}
\begin{split}
&s_j 
< 
\max \left\{ \frac{n}{p_j}, \frac{n}{2} \right\},
\quad
j=1, \dots, N,
\quad
\text{and}
\quad
s 
> 
-\max \left\{ \frac{n}{p^{\prime}}, \frac{n}{2} \right\},
\end{split}
\end{align}
then 
\begin{equation}\label{boundedness_1}
\Op(S^{m}_{0,0}(\R^n, N))
\subset
B(B^{s_1}_{p_1, q_1} \times \dots \times B^{s_N}_{p_N, q_N} \to B^s_{p, q}).
\end{equation}
\end{thm}


The condition  \eqref{sassum} is sharp in the following sense.
\begin{thm}\label{thmnec}
Let $0 < p, p_1, \dots, p_N \le \infty$, $0< q, q_1, \dots, q_N \le \infty$ and $s, s_1, \dots, s_N \in \R$.
If the boundedness \eqref{boundedness_1} holds with $m$ given in \eqref{criticalorder}, then
$s_j \le \max \{n/p_j, n/2\}$, $j= 1, \dots, N$,
and
$s \ge -\max \{n/p^{\prime}, n/2\}$.
\end{thm}

We shall explain some connection between our main results and  previous results.

Firstly, Theorem \ref{main1} yields that, 
for $0< p, p_1, \dots, p_N \le \infty$, $1/p \le 1/p_1 + \dots + 1/p_N$ and 
$0< q, q_1, \dots, q_N \le \infty$, $1/q \le 1/q_1 + \dots +1/q_N$, 
the boundedness
\begin{align}\label{bdd-main1-part}
\Op(S^{m}_{0,0}(\R^n, N))
\subset
B(B^{0}_{p_1, q_1} 
\times \dots \times 
B^{0}_{p_N, q_N} 
\to 
B^{0}_{p, q})
\end{align}
holds with
\begin{align*}
m 
= 
\min\left\{\frac{n}{p}, \frac{n}{2}\right\}
- 
\sum_{j=1}^N
\max\left\{\frac{n}{p_j}, \frac{n}{2}\right\}.
\end{align*}

If $0< p < \infty$ and $0 < p_1, \dots, p_N \le \infty$ satisfy
\begin{align}\label{addcondforexpo}
\frac{1}{\min\{p, 2\}}
\le 
\sum_{j=1}^N
\frac{1}{\max\{p_j, 2\}},
\end{align}
then the boundedness \eqref{bdd-main1-part} improves the boundedness result in Theorem C.
In fact, for $0< p < \infty$ and $0< p_1, \dots, p_N \le \infty$ satisfying \eqref{addcondforexpo}, we can choose $q_j = \max\{p_j, 2\}$, $j=1, \dots, N$, and $q = \min\{p, 2\}$ in \eqref{bdd-main1-part}, and hence we obtain 
\[
\Op(S^{m}_{0,0}(\R^n, N))
\subset
B(B^{0}_{p_1, \max\{p_1, 2\}} 
\times \dots \times 
B^{0}_{p_N, \max\{p_N, 2\}} 
\to 
B^{0}_{p, \min\{p, 2\}}).
\]
Since we have the embedding relations 
$B^{0}_{r, \min\{r, 2\}} \hookrightarrow h^r \hookrightarrow B^{0}_{r, \max\{r, 2\}}$ 
for $0 < r < \infty$,
and
$bmo \hookrightarrow B^{0}_{\infty, \infty}$, 
this boundedness gives an improvement of the corresponding $h^{p_1} \times \dots \times h^{p_N} \to h^p$ boundedness given in Theorem C. 

For the case $p= \infty$, 
if $0< p_1, \dots, p_N \le \infty$ satisfy
\[
1 
\le
\sum_{j=1}^N
\frac{1}{\max\{p_j, 2\}}, 
\]
then we can take $q_j = \max\{p_j, 2\}$, $j=1, \dots, N$ and $q =1$,
and consequently we obtain
\begin{align*}
&\Op(S^m_{0,0}(\R^n, N))
\subset
B(B^0_{p_1, \max\{p_1, 2\}} \times \dots \times B^0_{p_N, \max\{p_N, 2\}} \to B^0_{\infty, 1}).
\end{align*}
This improves the corresponding boundedness results given in Theorem C 
since $h^r \hookrightarrow B^{0}_{r, \max\{r, 2\}}$, $0< r < \infty$, $bmo \hookrightarrow B^0_{\infty, \infty}$ and 
$B^0_{\infty, 1}\hookrightarrow L^{\infty}$.

\bigskip
Secondly, the condition \eqref{sassum} is peculiar to the multilinear case. 
In fact, we can take any $s$ and $t$ in Theorem B, however, in the multilinear settings, Theorem \ref{thmnec} says that the conditions \eqref{sassum} are (almost) necessary
to assure the boundedness on Besov spaces.
We also notice that the condition \eqref{sassum} can be found in the author's paper \cite{Shida-Sobolev}.
In \cite{Shida-Sobolev}, it is proved that bilinear pseudo-differential operators with symbols in $S^m_{0, 0}(\R^n, 2)$ with the critical $m$ are bounded on Sobolev spaces under the assumption \eqref{sassum} with $N=2$.


The organization of this paper is as follows.
In Section 2, 
we give some notations and recall the definitions of some function spaces and embedding relations between them.
In Section 3, we give the proof of Theorem \ref{main1}.
In Section 4, we prove Theorem \ref{thmnec}.

\section{Preliminaries}\label{preliminaries}

For two nonnegative quantities $A$ and $B$,
the notation $A \lesssim B$ means that
$A \le CB$ for some unspecified constant $C>0$,
and $A \approx B$ means that
$A \lesssim B$ and $B \lesssim A$.


For $0 < p \le \infty$,
$p^{\prime}$ is  the conjugate exponent of $p$,
that is, $p^{\prime}$ is defined by 
$1/p+1/p^{\prime}=1$ if $1 < p \le \infty$ 
and 
$p^{\prime} = \infty$ if $0 < p \le 1$.



For a finite set $\Lambda$, 
$|\Lambda|$ denotes the number of the elements of $\Lambda$.


Let $\Sh(\R^n)$ and $\Sh'(\R^n)$ be the Schwartz space of
rapidly decreasing smooth functions on $\R^n$ and its dual,
the space of tempered distributions, respectively.
We define the Fourier transform $\F f$
and the inverse Fourier transform $\F^{-1}f$
of $f \in \Sh(\R^n)$ by
\[
\F f(\xi)
=\widehat{f}(\xi)
=\int_{\R^n}e^{-i \xi \cdot x} f(x)\, dx
\quad \text{and} \quad
\F^{-1}f(x)
=\frac{1}{(2\pi)^n}
\int_{\R^n}e^{i x \cdot \xi} f(\xi)\, d\xi.
\]
For $m \in L^{\infty}(\R^n)$,
the Fourier multiplier operator $m(D)$ is defined by
$m(D)f=\F^{-1}[m\widehat{f}]$ for $f \in \Sh(\R^n)$.

For a countable set $J$, the sequence space $\ell^q(J)$, $0 < q \le \infty$, is defined to be the set 
of all complex sequences $a = \{a_j\}_{j \in J}$ such that
\begin{align*}
\|a\|_{\ell^q(J)} 
=
\begin{cases}
\left(\sum_{j \in J} |a_j|^q \right)^{1/q}
&\text{if $0 < q < \infty$},
\\
\sup_{j \in J} |a_j|
&\text{if $q = \infty$}
\end{cases}
\end{align*}
is finite. For $a = \{a_j\}_{j \in J}$, we will use the notation $\|a_j\|_{\ell^q_j(J)}$ instead of $\|a\|_{\ell^q(J)}$
when we indicate the variable explicitly.

Let $\phi \in \Sh(\R^n)$ be such that $\int_{\R^n} \phi(x)\, dx \neq 0$.
For $0< p \le \infty$, the local Hardy space $h^p = h^p(\R^n)$ consists of all $f \in \Sh^\prime(\R^n)$ such that
\begin{equation*}
\|f\|_{h^p}
=
\left\|
\sup_{0< t < 1}
|\phi_t * f|
\right\|_{L^p}
< 
\infty,
\end{equation*}
where $\phi_t(x) = t^{-n} \phi(t^{-1} x)$. 
It is known that the definition of $h^p$ is independent of the choice of the function $\phi$
up to equivalence of quasi-norms.
It is also known
that $h^p = L^p$ for $1 < p \le \infty$ 
and $h^1 \hookrightarrow L^1$. 

The space $bmo = bmo(\R^n)$ consists of all locally integrable functions $f$ on $\R^n$
such that
\begin{equation*}
\|f\|_{bmo}
=
\sup_{|Q| \le 1}
\frac{1}{|Q|}
\int_{Q}
|f(x) -f_Q|
\, dx
+
\sup_{|Q| \ge 1}
\frac{1}{|Q|}
\int_{Q}
|f(x)|
\, dx
< \infty, 
\end{equation*}
where $Q$ ranges over all cubes in $\R^n$.
It is known that $L^\infty \subset bmo$.
It is also known that the dual space of $h^1$ coincides with $bmo$.

We recall the definition of Besov spaces.
Let $\psi_k \in \Sh(\R^n), k \ge 0,$ be such that 
\begin{align} \label{partition}
\begin{split}
&\supp \psi_0 \subset \{\xi \in \R^n : |\xi| \le 2\},
\quad
\supp \psi_k \subset \{\xi \in \R^n : 2^{k-1} \le |\xi| \le 2^{k+1}\},
\quad
k \ge 1,
\\
&
\|\partial^\alpha \psi_k\|_{L^\infty}
\le
C_\alpha 2^{-k |\alpha|},
\quad
\alpha \in \N^n_0,\ k \ge 0,
\\
&
\sum_{k = 0}^\infty
\psi_{k}(\xi)
=
1,
\quad
\xi \in \R^n.
\end{split}
\end{align} 
The Besov space
 $B^s_{p, q} = B^s_{p, q}(\R^n)$,
$0< p, q \le \infty$, 
$s \in \R$,
is defined to be the set of all 
$f \in \Sh^\prime(\R^n)$ 
such that 
\begin{align*}
\|f\|_{B^s_{p, q}}
=
\left\|
2^{k s} 
\left\|
\psi_k(D)f(x)
\right\|_{L^p_x(\R^n)} 
\right\|_{\ell^q_k(\N_0)}
<
\infty.
\end{align*}
It is known that the definition of Besov spaces 
is independent of the choice of $\psi_k$, $k=0, 1, 2, \dots$, up to the equivalence of quasi-norms.
If $1\le p, q < \infty$, then 
the dual space of $B^s_{p, q}$ coincides with 
$B^{-s}_{p^\prime, q^\prime}$.
The following embedding relations are well known;
\begin{align}
&B^s_{p, q_1} \hookrightarrow B^s_{p, q_2}, 
\label{Bq1Bq2}
\quad
\text{if}
\quad
q_1 \le q_2,
\\
& B^0_{p, \min\{p, 2\}} 
\hookrightarrow
h^p 
\hookrightarrow
B^0_{p, \max\{p, 2\}},
\quad
\text{if}
\quad
0<p<\infty,
\label{BhB}
\\
&
B^{0}_{\infty, 1} 
\hookrightarrow
L^{\infty}
\hookrightarrow
B^{0}_{\infty, \infty},
\label{BLBinfty}
\\
&
bmo
\hookrightarrow
B^{0}_{\infty, \infty}.
\notag
\end{align}
As a consequence of \eqref{Bq1Bq2}, \eqref{BhB} and \eqref{BLBinfty}, we have $h^p \hookrightarrow B^0_{p, \infty}$,
$0 < p \le \infty$,
which means
\begin{equation} \label{embd-hpB0pinfty}
\sup_{k \in \N_0}
\|\psi_k(D)f\|_{L^p}
\lesssim
\|f\|_{h^p}.
\end{equation}
For more basic properties about Besov spaces, 
see, e.g., Triebel \cite{Triebel-ToFS}.

It is known that
the $L^p$-norm in the definition of $B^s_{p, q}$-norm 
can be replaced by the $h^p$-norm. 
More precisely, the following proposition was given by Qui \cite{Qui}. 

\begin{prop}[\cite{Qui}]\label{propQui}
Let $0< p, q \le \infty$ and $s \in \R$. Then,
\begin{equation}\label{Bhpequiv}
\|f\|_{B^s_{p, q}}
\approx
\left\|
2^{k s} 
\left\|
\psi_k(D)f(x)
\right\|_{h^p_x(\R^n)}
\right\|_{\ell^q_{k}(\N_0)}.
\end{equation}
\end{prop}

We end this section by recalling the definition and some properties of  the Wiener amalgam space.
Let
$\kappa \in \Sh(\R^n)$ be such that $\supp \kappa$ is compact and
\begin{align} \label{part-Wiener}
\left|
\sum_{\mu \in \Z^n}
\kappa(\xi - \mu)
\right|
\ge 1,
\quad
\xi \in \R^n.
\end{align}
The Wiener amalgam space $W^{p, q}_s =W^{p, q}_s(\R^n)$, $0< p, q \le \infty$, $s \in\R$, 
consists of all $f \in \Sh^\prime(\R^n)$ such that 
\[
\|f\|_{W^{p, q}_s}
=
\left\|
\left\|
\la \mu \ra^{s}
\Box_{\mu}f(x)
\right\|_{\ell^q_{\mu}(\Z^n)}
\right\|_{L^p_x(\R^n)}
<
\infty,
\]
where $\Box_{\mu}f= \kappa(D-\mu)f = \F^{-1}[\kappa(\cdot-\mu) \widehat{f}]$.
We simply write $W^{p, q} = W^{p, q}_{0}$.
The space $W^{p, q}_{s}$ does not depend on the choice of $\kappa$ up to equivalence of quasi-norms.
For the definition of the Wiener amalgam space,
see Triebel \cite{Triebel-ZAA}. 

The embedding relations between Lebesgue, local Hardy spaces 
and Wiener amalgam spaces are well investigated as follows.

\begin{lem}[\cite{CKS, GWYZ}] \label{embd}
Let $0< p, p_1, p_2, q_1, q_2 \le \infty$. Then,
\begin{align}
&W^{p_1, q_1}_s \hookrightarrow W^{p_2, q_2}_s
\quad
\text{if}
\quad
p_1 \le p_2, 
\quad
\text{and}
\quad
q_1 \le q_2;
\label{monotone}
\\
&W^{p, 2}_{\alpha(p)} \hookrightarrow h^p 
\quad
\text{if}
\quad
0< p< \infty,
\quad
\text{where}
\quad
\alpha(p) 
= 
(n/2)
-
\min
\{n/2, n/p\};
\label{Wh}
\\
&h^p \hookrightarrow W^{p, 2}_{\beta(p)}
\quad
\text{if}
\quad
0< p< \infty,
\quad
\text{where}
\quad
\beta(p) 
= 
(n/2)
-
\max
\{n/2, n/p\};
\label{hW}
\\
&W^{\infty, 1} \hookrightarrow L^\infty.
\label{WLinfty}
\end{align}
\end{lem}
The embedding relations \eqref{Wh} and \eqref{hW} are given 
by Cunanan-Kobayashi-Sugimoto \cite{CKS} for the case $1< p< \infty$
and Guo-Wu-Yang-Zhao\cite{GWYZ} for the $0< p \le 1$.
The embedding \eqref{WLinfty} is given by \eqref{WLinfty}
The proof of \eqref{monotone} can be found in Kato-Miyachi-Tomita \cite{KMT-JFA}.

We will use these embedding relations in the proof of Theorem \ref{main1}.
The idea of using the Wiener amalgam spaces comes from the recent works of T. Kato, A. Miyachi and N. Tomita (see \cite{KMT-JPDOA, KMT-JMSJ, KMT-JFA}).

In the rest of this section, we shall show the following proposition.
\begin{prop} \label{Keyprop}
Let $N \ge 2$, $0 < p_0, p_1, \dots, p_N \le \infty$, $1/p_0 = 1/p_1 + \dots + 1/p_N$, and $M_0 \in \N_0$. 
Let $R, R_1, \dots, R_N \ge 1$.
Let $\Lambda, \Lambda_1, \dots, \Lambda_N$ be subsets of $\Z^n$ satisfying
\begin{align*}
\Lambda = \{\nu \in \Z^n : |\nu| \lesssim R\},
\quad
\Lambda_j  = \{\nu \in \Z^n : |\nu| \approx R_j\},
\quad
j=1, \dots, N.
\end{align*}
Suppose that $R_1 = \max_{1 \le j \le N} R_j$ and that $R_2= \max_{2 \le j \le N} R_j$.
Then the estimate
\begin{align*}
&
\left\|
\left\|
\sum_{\substack{ \Nu \in \Lambda_1 \times \dots \times \Lambda_N
\\
\nu_1 + \dots + \nu_N = \tau}
}
\prod_{j=1}^N
|\Box_{\nu_j} f_j|
\right\|_{\ell^2_\tau(\Lambda)}
\right\|_{L^{p_0}}
\lesssim
\min \{R_2^{n/2}, R^{n/2}\}
\prod_{j=3}^N
R_j^{n/2}
\prod_{j=1}^N
R_j^{-\beta(p_j)}
\|f_j\|_{h^{p_j}}
\end{align*}
holds with the implicit constant independent of $R_1, \dots, R_N$ and $R$.
\end{prop}

\begin{proof}
Notice that $|\Lambda| \lesssim R^n$ and $|\Lambda_j| \approx R_j^n$, $j=1, \dots, N$.
First, 
we have by Young's inequality and H\"older's inequality
\begin{align*}
\left\|
\sum_{\substack{ 
\Nu \in \Lambda_1 \times \dots \times \Lambda_N
\\
\nu_1 + \dots + \nu_N = \tau
}}
\prod_{j=1}^N
|\Box_{\nu_j} f_j|
\right\|_{\ell^2_\tau(\Lambda)}
&\le
\left\|
\sum_{\substack{ 
\Nu \in \Lambda_1 \times \dots \times \Lambda_N
\\
\nu_1 + \dots + \nu_N = \tau
}}
\prod_{j=1}^N
|\Box_{\nu_j} f_j|
\right\|_{\ell^2_\tau(\Z^n)}
\\
&\le
\left\|
\Box_{\nu_1} f_1
\right\|_{\ell^2_{\nu_1}(\Lambda_1)}
\prod_{j=2}^N
\left\|
\Box_{\nu_j} f_j
\right\|_{\ell^1_{\nu_j}(\Lambda_j)}
\\
&\lesssim
\left\|
\Box_{\nu_1} f_1
\right\|_{\ell^2_{\nu_1}(\Lambda_1)}
\prod_{j=2}^N
R_j^{n/2}
\left\|
\Box_{\nu_j} f_j
\right\|_{\ell^2_{\nu_j}(\Lambda_j)}.
\end{align*}
Hence, this estimate and H\"older's inequality yield that
\begin{align} \label{est-R2}
\left\|
\left\|
\sum_{\substack{ 
\Nu \in \Lambda_1 \times \dots \times \Lambda_N
\\
\nu_1 + \dots + \nu_N = \tau
}}
\prod_{j=1}^N
|\Box_{\nu_j} f_j|
\right\|_{\ell^2_\tau(\Lambda)}
\right\|_{L^{p_0}}
\lesssim
\prod_{j = 2}^N
R_j^{n/2}
\prod_{j=1}^N
\left\|
\left\|
\Box_{\nu_j} f_j
\right\|_{\ell^2_{\nu_j}(\Lambda_j)}
\right\|_{L^{p_j}}.
\end{align}
On the other hand,
 it follows from H\"older's inequality  that
\begin{align*}
\sum_{\substack{ 
\Nu \in \Lambda_1 \times \dots \times \Lambda_N
\\
\nu_1 + \dots + \nu_N = \tau
}}
\prod_{j=1}^N
|\Box_{\nu_j} f_j|
&=
\sum_{\nu_1 \in  \Lambda_1}
|\Box_{\nu_1} f_1|
\sum_{
\substack{
(\nu_2, \dots, \nu_N) \in \Lambda_2 \times \dots \times \Lambda_N
\\
 \nu_2 + \dots + \nu_N = \tau -\nu_1}
}
\prod_{j=2}^N
|\Box_{\nu_j} f_j|
\\
&\le
\left\|
\Box_{\nu_1} f_1
\right\|_{\ell^2_{\nu_1}(\Lambda_1)}
\left\|
\sum_{
\substack{
(\nu_2, \dots, \nu_N) \in \Lambda_2 \times \dots \times \Lambda_N
\\
 \nu_2 + \dots + \nu_N = \tau - \nu_1}
}
\prod_{j=2}^N
|\Box_{\nu_j} f_j|
\right\|_{\ell^2_{\nu_1}(\Z^n)}
\\
&=
\left\|
\Box_{\nu_1} f_1
\right\|_{\ell^2_{\nu_1}(\Lambda_1)}
\left\|
\sum_{
\substack{
(\nu_2, \dots, \nu_N) \in \Lambda_2 \times \dots \times \Lambda_N
\\
 \nu_2 + \dots + \nu_N = \mu}
}
\prod_{j=2}^N
|\Box_{\nu_j} f_j|
\right\|_{\ell^2_{\mu}(\Z^n)}.
\end{align*}
By Young's inequality and H\"older's inequality, we have
\begin{equation*} 
\left\|
\sum_{
\substack{
(\nu_2, \dots, \nu_N) \in \Lambda_2 \times \dots \times \Lambda_N
\\
 \nu_2 + \dots + \nu_N = \mu}
}
\prod_{j=2}^N
|\Box_{\nu_j} f_j|
\right\|_{\ell^2_{\mu}(\Z^n)}
\lesssim
\prod_{j=3}^N
R_j^{n/2}
\prod_{j=2}^N
\left\|
\Box_{\nu_j} f_j
\right\|_{\ell^2_{\nu_j}(\Lambda_j)}.
\end{equation*}
Hence we obtain by H\"older's inequality
\begin{align} \label{est-R}
\left\|
\left\|
\sum_{\substack{ 
\Nu \in \Lambda_1 \times \dots \times \Lambda_N
\\
\nu_1 + \dots + \nu_N = \tau
}}
\prod_{j=1}^N
|\Box_{\nu_j} f_j|
\right\|_{\ell^2_\tau(\Lambda)}
\right\|_{L^{p_0}}
&\lesssim
R^{n/2}
\prod_{j=3}^N
R_j^{n/2}
\prod_{j=1}^N
\left\|
\left\|
\Box_{\nu_j} f_j
\right\|_{\ell^2_{\nu_j}(\Lambda_j)}
\right\|_{L^{p_j}}.
\end{align}
Therefore, combining \eqref{est-R2} and \eqref{est-R}, we obtain
\[
\left\|
\left\|
\sum_{\substack{ 
\Nu \in \Lambda_1 \times \dots \times \Lambda_N
\\
\nu_1 + \dots + \nu_N = \tau
}}
\prod_{j=1}^N
|\Box_{\nu_j} f_j|
\right\|_{\ell^2_\tau(\Lambda)}
\right\|_{L^{p_0}}
\lesssim
\min \{R_2^{n/2}, R^{n/2}\}
\prod_{j=3}^N
R_j^{n/2}
\prod_{j=1}^N
\left\|
\left\|
\Box_{\nu_j} f_j
\right\|_{\ell^2_{\nu_j}(\Lambda_j)}
\right\|_{L^{p_j}}.
\]
Since $\la \nu_j \ra \approx R_j$ if $\nu_j \in \Lambda_{j}$, it follows from the embedding $h^{p_j} \hookrightarrow W^{p_j, 2}_{\beta(p_j)}$ that 
\begin{align*}
\left\|
\left\|
\Box_{\nu_j} f_j
\right\|_{\ell^2_{\nu_j}(\Lambda_j)}
\right\|_{L^{p_j}}
&\lesssim
R_j^{-\beta(p_j)}
\left\|
\left\|
\langle \nu_j \rangle^{\beta(p_j)}
\Box_{\nu_j} f_j
\right\|_{\ell^2_{\nu_j}(\Z^n)}
\right\|_{L^{p_j}}
\\
&
\lesssim
R_j^{-\beta(p_j)}
\|f_j\|_{h^{p_j}},
\quad
j=1, \dots, N.
\end{align*}
The proof is complete.
\end{proof}

\begin{rem} \label{Keyrem}
By the Cauchy-Schwartz inequality and Proposition \ref{Keyprop}, we also have 
\begin{align*}
\left\|
\left\|
\sum_{\substack{ \Nu \in \Lambda_1 \times \dots \times \Lambda_N
\\
\nu_1 + \dots + \nu_N = \tau}
}
\prod_{j=1}^N
|\Box_{\nu_j} f_j|
\right\|_{\ell^1_\tau(\Lambda)}
\right\|_{L^{p_0}}
\lesssim
R^{n/2}
\min \{R_2^{n/2}, R^{n/2}\}
\prod_{j=3}^N
R_j^{n/2}
\prod_{j=1}^N
R_j^{-\beta(p_j)}
\|f_j\|_{h^{p_j}}.
\end{align*}
We also use this estimate in the proof of Theorem \ref{main1}.
\end{rem}

\section{Proof of Theorem \ref{main1}}

In this section, we shall prove Theorem \ref{main1}.
Let $0< p, p_j, q, q_j \le \infty$ and $s, s_j \in \R$, $j=1, \dots, N$, be the same as in Theorem \ref{main1}.
Throughout this section, 
we always assume that  $\sigma \in S^{m}_{0,0}(\R^n, N)$ with $m$ given by \eqref{criticalorder}.
We use the notation $\Xxi = (\xi_1, \dots, \xi_N) \in (\R^n)^N$.

The following method using the Fourier series expansion 
goes back at least to Coifman-Meyer \cite{CM-Asterisque, CM-AIF}. 

Let $\vphi, \widetilde{\vphi} \in \Sh(\R^n)$ be such that
\begin{align*}
&\supp \vphi \subset [-1, 1]^n, 
\quad 
\sum_{\nu \in \Z^n} \vphi(\xi-\nu) = 1, 
\quad 
\xi \in \R^n,
\\
&\supp \widetilde{\vphi} \subset [-3, 3]^n,
\quad
0 \le \widetilde{\vphi} \le 1,
\quad
\widetilde{\vphi} = 1 
\quad
\text{on}
\quad
[-1, 1]^n.
\end{align*}
We remark that $\vphi$ and $\widetilde{\vphi}$
satisfy \eqref{part-Wiener}.

We decompose the symbol $\sigma = \sigma(x, \Xxi)$ as
\begin{align*}
\sigma(x, \Xxi)
&=
\sum_{\Nu = (\nu_1, \dots, \nu_N) \in (\Z^n)^N} 
\sigma(x, \Xxi)
\prod_{j=1}^N
\vphi(\xi_j-\nu_j)
=
\sum_{\Nu \in (\Z^n)^N} 
\sigma_{\Nu}(x, \Xxi),
\end{align*}
where $\sigma_{\Nu}(x, \Xxi) = \sigma(x, \Xxi) \prod_{j=1}^N \vphi(\xi_j-\nu_j)$.
We define
\[
S_{\Nu}(x, \Xxi)
=
\sum_{\bl \in (\Z^n)^N} 
\sigma_{\Nu} (x, \Xxi-2\pi \bl).
\]
Since $S_{\Nu}(x, \Xxi) = \sigma_{\Nu}(x, \Xxi)$ if $\Xxi \in \Nu + [-3, 3]^{Nn}$, we have
\[
\sigma_{\Nu}(x, \Xxi)
=
S_{\Nu}(x, \Xxi)
\prod_{j=1}^N
\widetilde{\vphi}(\xi_j-\nu_j)
\]
Furthermore, since $S_{\Nu}$ is a $2\pi \Z^{Nn}$-periodic function with respect to the $\Xxi$-variable, 
the Fourier series expansion yields that 
\[
\sigma_{\Nu}(x, \Xxi)
=
\sum_{\Mu = (\mu_1, \dots, \mu_N) \in (\Z^n)^N}
P_{\Nu, \Mu}(x)
\prod_{j=1}^N
e^{i \mu_j \cdot \xi_j}
\widetilde{\vphi}(\xi_j-\nu_j),
\]
where
\begin{align*}
P_{\Nu, \Mu}(x)
=
\frac{1}{(2\pi)^{Nn}}
\int_{\Nu + [-\pi, \pi]^{Nn}}
e^{-i \Mu \cdot \Y}
\sigma_{\Nu}(x, \Y)
\,
d\Y.
\end{align*}
It follows from  integration by parts that
\begin{align*}
&P_{\Nu, \Mu}(x)
=
\la \Mu \ra^{-2M}
Q_{\Nu, \Mu}(x),
\end{align*}
where 
\begin{align*}
&Q_{\Nu, \Mu}(x)
=
\frac{1}{(2\pi)^{Nn}}
\int_{\Nu + [-\pi, \pi]^{Nn}}
e^{-i \Mu \cdot \Y}
(I-\Delta_{\Y})^{M}
\sigma_{\Nu}(x, \Y)
\,
d\Y.
\end{align*}
We remark that, for $\alpha \in \N^n_0$, 
\begin{equation}\label{estPnu}
|\pa^\alpha_x Q_{\Nu, \Mu}(x)|
\lesssim
\la \Nu \ra^{m}
\end{equation}
holds for all $\Nu, \Mu \in (\Z^n)^N$,
since $\sigma_{\Nu}$ satisfies
\[
|
\pa_x^{\alpha}
\pa_{\Xxi}^{\boldsymbol{\beta}}
\sigma_{\Nu}(x, \Xxi)
|
\le
C_{\alpha, \boldsymbol{\beta}}
\la \Nu \ra^m,
\quad
\alpha \in \N_0^n,\ \  \boldsymbol{\beta} \in (\N_0^n)^N.
\]
Thus we can write $T_{\sigma}$ as 
\[
T_{\sigma}(f_1, \dots, f_N)(x)
=
\sum_{\Mu \in (\Z^n)^N}
\la \Mu \ra^{-2M}
\sum_{\Nu \in (\Z^n)^N}
Q_{\Nu, \Mu}(x)
\prod_{j=1}^N
\Box_{\nu_j}f_j(x+ \mu_j)
\]
Choosing the number $M$ as large as $2M \min \{1, p, q\}>Nn$, we obtain
\[
\left\| 
T_{\sigma}(f_1, \dots, f_N)
\right\|_{B^{s}_{p, q}}
\lesssim
\sup_{\Mu \in (\Z^n)^N}
\left\|
\sum_{\Nu \in (\Z^n)^N}
Q_{\Nu, \Mu}
\prod_{j=1}^N
\Box_{\nu_j}f_j(\cdot + \mu_j)
\right\|_{B^s_{p, q}}.
\]
Let $\psi_{\ell_j} \in \Sh(\R^n),\ \ell_j \in \N_0$, $j=0, 1, \dots, N$, be the same partition of unities as in the definition of Besov spaces.
We further decompose the sum on the right hand side above as follows; 
\begin{align*}
&\sum_{\Nu \in (\Z^n)^N}
Q_{\Nu, \Mu}(x)
\prod_{j=1}^N
\Box_{\nu_j}f_j(x+ \mu_j)
\\
&=
\sum_{\bl = (\ell_0, \ell_1, \dots, \ell_N) \in (\N_0)^{N+1}}
\sum_{\Nu \in (\Z^n)^N}
\psi_{\ell_0}(D)Q_{\Nu, \Mu}(x)
\prod_{j=1}^N
\Box_{\nu_j}
\psi_{\ell_j}(D)f_j(x+ \mu_j)
\\
&=
\sum_{\bl \in (\N_0)^{N+1}}
\sum_{\Nu \in (\Z^n)^N}
Q_{\ell_0, \Nu, \Mu}(x)
\prod_{j=1}^N
\Box_{\nu_j} F^{j}_{\ell_j, \mu_j}(x),
\end{align*}
where we set $Q_{\ell_0, \Nu, \Mu} = \psi_{\ell_0}(D)Q_{\Nu, \Mu}$ and $F^{j}_{\ell_j, \mu_j} = \psi_{\ell_j}(D)f_j(\cdot + \mu_j)$, $j=1, \dots, N$.

Now, we divide the sum with respect to the variable $\bl$ into the following $N$ parts:

\begin{align*}
&
\Lambda_1
= 
\left\{\bl \in (\N_0)^{N+1} 
: 
\ell_j \le \ell_1, \  
j=2, \dots, N \right\},
\\
&
\Lambda_k
= 
\left\{\bl \in (\N_0)^{N+1} 
:
\begin{array}{l}
\ell_j < \ell_k, \  j= 1, \dots, k-1,
\\ 
\ell_j \le \ell_k, \ j = k+1, \dots, N
\end{array}
\right\},
\quad
k=2, \dots, N-1,
\\
&
\Lambda_{N} = \left\{\bl \in (\N_0)^{N+1} : \ell_j < \ell_N, \ j=1, \dots, N-1 \right\}.
\end{align*}
By symmetry, it is sufficient to deal with the sum concerning with $\Lambda_1$. 
Furthermore, we divide the set $\Lambda_1$ into the following three parts;
\begin{align*}
\Lambda_1
=
&\big\{
\bl \in \Lambda_1
:
\ell_0 \ge \ell_1-3
\big\}
\\
&\cup
\big\{
\bl \in \Lambda_1
:
\ell_0 \le \ell_1-4,
\quad
\max \{\ell_2, \dots, \ell_N\} \le \ell_1-N-2
\big\}
\\
&\cup
\big\{
\bl \in \Lambda_1
:
\ell_0 \le \ell_1-4, 
\quad
\max\{ \ell_2, \dots, \ell_N \}
\ge
\ell_1-N-1
\big\}.
\end{align*}
By symmetry, it is sufficient to consider the case $\max\{\ell_2, \dots, \ell_N\} = \ell_2$.
In particular we may assume that $\ell_j \le \ell_2$, $j=3, \dots, N$.

Summarizing the above observations, 
it is sufficient to prove that the estimates
\begin{equation}\label{GOAL!!!!}
S_i :=
\left\|
\sum_{\bl \in D_{i}}
\sum_{\Nu \in (\Z^n)^N}
Q_{\ell_0, \Nu, \Mu}
\prod_{j=1}^N
\Box_{\nu_j} F^j_{\ell_j, \mu_j}
\right\|_{B^{s}_{p, q}}
\lesssim
\prod_{j=1}^N
\|f_{j}\|_{B^{s_j}_{p_j, q_j}},
\quad
i=1, 2, 3,
\end{equation}
hold with the implicit constant independent of $\Mu$, where
\begin{align*}
&
D_{1}
=
\big\{
\bl \in (\N_0)^{N+1}
:
\ell_0 \ge \ell_1-3,
\quad
\ell_j \le \ell_1, \ 
j=2, \dots, N
\big\},
\\
&
D_{2}
=
\left\{
\bl \in (\N_0)^{N+1}
:
\ell_0 \le \ell_1-4, 
\quad
\ell_2 \le \ell_1-N-2,
\quad
\ell_j \le \ell_2 \le \ell_1, \ 
j=3, \dots, N
\right\},
\\
&
D_{3}
=
\left\{
\bl \in (\N_0)^{N+1}
:
\ell_0 \le \ell_1-4, 
\quad
\ell_2 \ge \ell_1-N-1,
\quad
\ell_j \le \ell_2 \le \ell_1,\ 
j=3, \dots, N
\right\}.
\end{align*}

\begin{lem}\label{EST-x}
Let  
$m \in \R$
and $L \in \N_0$. 
If  $\sigma \in S^{m}_{0,0}(\R^n, N)$, then
\[
\|Q_{\ell_0, \Nu, \Mu}\|_{L^\infty}
\lesssim
2^{-\ell_0L}
\la \Nu \ra^{m}
\]
holds for all $\Nu, \Mu \in (\Z^n)^N$ and $\ell_0 \in \N_0$.
\end{lem}

\begin{proof}
We first consider the case $\ell_0 \ge 1$. 
Since $\F^{-1}\psi_{\ell_0}$ has the moment condition, Taylor's expansion gives that 
\begin{align*}
Q_{\ell_0, \Nu, \Mu}(x)
&=
\int_{\R^n}
\F^{-1}\psi_{\ell_0}(y) Q_{\Nu, \Mu}(x-y)
\,
dy
\\
&=
\int_{\R^n}
\F^{-1}\psi_{\ell_0}(y) 
\left(
Q_{\Nu, \Mu}(x-y)
-
\sum_{|\alpha| \le L-1}
\frac{(-y)^\alpha}{\alpha !}
\pa^\alpha Q_{\Nu, \Mu}(x)
\right)
\,
dy
\\
&=
\int_{\R^n}
\F^{-1}\psi_{\ell_0}(y) 
\left(
L
\sum_{|\alpha| = L}
\frac{(-y)^\alpha}{\alpha !}
\int_0^1
(1-t)^{L-1}
[\pa^\alpha Q_{\Nu, \Mu}](x-ty)
\,
dt
\right)
\,
dy
\end{align*}
Hence, it follows from \eqref{estPnu} that
\begin{align*}
|Q_{\ell_0, \Nu, \Mu}(x)|
&\lesssim
\int_{\R^n}
|\F^{-1}\psi_{\ell_0}(y)|
\left(
\sum_{|\alpha| = L}
|(-y)^\alpha|
\int_0^1
|\pa^\alpha Q_{\Nu, \Mu}(x-ty)|
\,
dt
\right)
\,
dy
\\
&\lesssim
\la \Nu \ra^{m}
\int_{\R^n}
\frac{2^{\ell_0 n}|y|^L}{(1+2^{\ell_0}|y|)^{L+n+\epsilon}}
\,
dy
\\
&\lesssim
2^{-\ell_0L}
\la \Nu \ra^{m}.
\end{align*}
Here, we used the estimate $|\F^{-1} \psi_{\ell_0}(x)| \lesssim 2^{\ell_0n}(1+2^{\ell_0}|x|)^{-(L+n+\epsilon)}$ in the second inequality.
We obtain the same estimate for $\ell_0=0$ without using  the moment condition. The proof is complete.
\end{proof}


Now, we shall prove the estimate \eqref{GOAL!!!!}.
In what follows, we use the notation $f_{j, k} = \psi_k(D)f_j$ for $j = 1, \dots, N$ and $k \in \N_0$.

Since 
\begin{align}
&\supp \F [Q_{\ell_0, \Nu, \Mu}] \subset \{|\xi| \le 2^{\ell_0+1}\},
\label{suppP}
\\
&\supp \F [\Box_{\nu_j} F^{j}_{\ell_j, \mu_j}] \subset \nu_j + [-3, 3]^n, \quad j=1, \dots, N,
\label{suppF}
\end{align}
we have
\begin{equation} \label{suppunif}
\supp \F
\left[
Q_{\ell_0, \Nu, \Mu}
\prod_{j=1}^N
\Box_{\nu_j} F^{j}_{\ell_j, \mu_j}
\right] 
\subset
\nu_1 + \dots + \nu_N 
+ 
\left[
-2^{\ell_0+d}, 2^{\ell_0+d}
\right]^n
\end{equation} 
for some $d = d_N > 0$ depending on $N$.
Thus we have
\begin{align} \label{diag-rest}
\begin{split}
&\psi_k(D)
\left[
\sum_{\bl \in D_{i}}
\sum_{\Nu \in (\Z^n)^N}
Q_{\ell_0, \Nu, \Mu}
\prod_{j=1}^N
\Box_{\nu_j} F^j_{\ell_j, \mu_j}
\right]
\\
&=
\psi_k(D)
\left[
\sum_{\bl \in D_i}
\sum_{\Nu : \nu_1 + \dots + \nu_N \in \Lambda_{k, \ell_0}}
Q_{\ell_0, \Nu, \Mu}
\prod_{j=1}^N
\Box_{\nu_j} F^{j}_{\ell_j, \mu_j}
\right]
\end{split}
\end{align}
with
\begin{align*}
\Lambda_{k, \ell_0}
=
\{
\nu \in \Z^n
\,
:
\,
\supp \psi_k
\cap
(\nu + [-2^{\ell_0+d}, 2^{\ell_0+d}]^{n})
\neq
\emptyset
\}.
\end{align*}
We remark that $|\nu| \lesssim 2^{\ell_0+k}$ if $ \nu \in \Lambda_{k, \ell_0}$, 
and consequently $|\Lambda_{k, \ell_0}| \lesssim 2^{(\ell_0+k)n}$.
We set 
\begin{align*}
R_{\bl, k}
=
R_{\bl, k, \Mu}
=
\sum_{\Nu : \nu_1 + \dots + \nu_N \in \Lambda_{k, \ell_0}}
Q_{\ell_0, \Nu, \Mu}
\prod_{j=1}^N
\Box_{\nu_j} F^{j}_{\ell_j, \mu_j}.
\end{align*}
For $M_0 \in \R$,  we now prove that 
the following estimate holds for all $\bl = (\ell_0, \ell_1, \dots, \ell_N) \in (\N_0)^{N+1}$ and $k \in \N_0$:
\begin{align}\label{Hulk}
&\left\|
R_{\bl, k}
\right\|_{h^p}
\lesssim
2^{-\ell_0 M_0}
2^{k \alpha(p)}
2^{\ell_1m}
2^{\min\{\ell_2, k \}n/2}
\prod_{j=3}^{N}
2^{\ell_j n/2}
\prod_{j=1}^N
2^{-\ell_j\beta(p_j)}
\|f_{j, \ell_j}\|_{h^{p_j}}.
\end{align}
Here the implicit constant does not depend on $\Mu$.

Firstly, we prove that the estimate \eqref{Hulk} holds with $0< p< \infty$. 
By the embedding $W^{p_0, 2}_{\alpha(p)} \hookrightarrow W^{p, 2}_{\alpha(p)} \hookrightarrow h^p$, we have 
\begin{align*}
\|R_{\bl, k}\|_{h^p}
&\lesssim
\|R_{\bl, k}\|_{W^{p_0, 2}_{\alpha(p)}}
=
\left\|
\left\|
\langle \tau \rangle^{\alpha(p)}
\Box_{\tau}
R_{\bl, k}
\right\|_{\ell^2_{\tau}(\Z^n)}
\right\|_{L^{p_0}}.
\end{align*}
Recalling that \eqref{suppP}, \eqref{suppF} and that the function $\phi$  has compact support (see the definition of $\Box_{\tau}$),
we write
\begin{align*}
&\|R_{\bl, k}\|_{h^{p}}
\lesssim
\left\|
\left\|
\la \tau \ra^{\alpha(p)}
\Box_{\tau}
\left[
\sum_{\substack{\Nu : \nu_1 + \dots + \nu_N \in \Lambda_{k, \ell_0}
\\
|\nu_1 + \dots + \nu_N - \tau| \lesssim 2^{\ell_0}
}
}
Q_{\ell_0, \Nu, \Mu}
\prod_{j=1}^N
\Box_{\nu_j} F^{j}_{\ell_j, \mu_j}
\right]
\right\|_{\ell^2_\tau(\Z^n)}
\right\|_{L^{p_0}}
\\
&\lesssim
2^{\ell_0 n/p_0}
\sum_{|\lambda| \lesssim 2^{\ell_0}}
\left\|
\left\|
\langle \tau-\lambda \rangle^{\alpha(p)}
\Box_{\tau-\lambda}
\left[
\sum_{\Nu : \nu_1 + \dots + \nu_N  = \tau}
Q_{\ell_0, \Nu, \Mu}
\prod_{j=1}^N
\Box_{\nu_j} F^{j}_{\ell_j, \mu_j}
\right]
\right\|_{\ell^2_\tau(\Lambda_{k, \ell_0})}
\right\|_{L^{p_0}}
\\
&\lesssim
2^{\ell_0 (\alpha(p) + n/p_0)}
2^{k \alpha(p)}
\sum_{|\lambda| \lesssim 2^{\ell_0}}
\left\|
\left\|
\Box_{\tau-\lambda}
R_\tau
\right\|_{\ell^2_k(\Lambda_{k, \ell_0})}
\right\|_{L^{p_0}},
\end{align*}
where
we set
\begin{equation} \label{Nanjakore}
R_\tau
=
\sum_{\Nu : \nu_1 + \dots + \nu_N = \tau}
Q_{\ell_0, \Nu, \Mu}
\prod_{j=1}^N
\Box_{\nu_j} F^{j}_{\ell_j, \mu_j}.
\end{equation}
and used the inequality $\la \tau -\lambda \ra \lesssim 2^{\ell_0+k}$ for
$|\lambda| \lesssim 2^{\ell_0}$ and $\tau \in \Lambda_{k, \ell_0}$.

Now, by recalling \eqref{suppunif}, 
we have $\supp \F[\Box_{k-\lambda}R_\tau] \subset \{|\zeta - \tau| \lesssim 2^{\ell_0}\}$.
Hence Nikol'skij's inequality (see, e.g, \cite[Remark 1.3.2/1]{Triebel-ToFS}) gives that
\begin{align} \label{pointwiseNikolskij}
|\Box_{k-\lambda} R_\tau(x)|
\le
\|\Phi(\cdot)R_\tau(x-\cdot)\|_{L^1}
\lesssim
2^{\ell_0 n (1/r_0 -1)}
\|\Phi(\cdot)R_\tau(x-\cdot)\|_{L^{r_0}},
\end{align}
where $\Phi = \F^{-1} \phi$ and $r_0 = \min \{1, p_0\}$.
By Minkowski's inequality, we have
\begin{align*}
\left\|
\left\|
\Box_{k-\lambda}
R_\tau(x)
\right\|_{\ell^2_\tau(\Lambda_{k, \ell_0})}
\right\|_{L^{p_0}_x}
&\lesssim
2^{\ell_0 n (1/r_0-1)}
\left\|
\left\|
\left\|
\Phi(y)R_\tau(x-y)
\right\|_{L^{r_0}_y}
\right\|_{\ell^2_\tau(\Lambda_{k, \ell_0})}
\right\|_{L^{p_0}_x}
\\
&\lesssim
2^{\ell_0 n (1/r_0-1)}
\left\|
\left\|
R_\tau(x)
\right\|_{\ell^2_\tau(\Lambda_{k, \ell_0})}
\right\|_{L^{p_0}_x}.
\end{align*}
Hence, by applying Lemma \ref{EST-x} with $L$ replaced by 
$M_0+\alpha(p) + n/p_0 +n/r_0 +n/2$, 
we obtain
\begin{align*}
\|R_{\bl, k}\|_{h^{p}}
&\lesssim
2^{-\ell_0(M_0+n/2)}
2^{k \alpha(p)}
\left\|
\left\|
\sum_{\Nu : \nu_1 + \dots + \nu_N  = \tau}
\la \Nu \ra^{m}
\prod_{j=1}^N
\left|
 \Box_{\nu_j} F^{j}_{\ell_j, \mu_j}
\right|
\right\|_{\ell^2_\tau(\Lambda_{k, \ell_0})}
\right\|_{L^{p_0}}.
\end{align*}
If $\supp \varphi(\cdot-\nu_j) \cap \supp \psi_{\ell_j} = \emptyset$, then $\Box_{\nu_j} F^{j}_{\ell_j, \mu_j} = 0$.
Hence the sum over the $\nu_j$-variable can be restricted to 
\[
\Lambda_{\ell_j}
=
\{\nu_j \in \Z^n
:
\supp \varphi(\cdot-\nu_j) 
\cap 
\supp \psi_{\ell_j} 
\neq 
\emptyset
\},
\quad
j=1, \dots, N.
\]
Notice that $|\nu_j| \approx 2^{\ell_j}$ if $\nu_j \in \Lambda_{\ell_j}$.
Furthermore, since $\la \Nu \ra \approx 2^{\ell_1}$ for all $\Nu = (\nu_1, \dots, \nu_N) \in \Lambda_{\ell_1} \times \dots \times \Lambda_{\ell_N}$
if $\ell_1 \ge \ell_j$, $j=2, \dots, N$,
we have
\begin{align*} 
\sum_{\Nu : \nu_1 + \dots + \nu_N = \tau}
\la \Nu \ra^{m}
\prod_{j=1}^N
\left|
\Box_{\nu_j} F^{j}_{\ell_j, \mu_j}
\right|
\lesssim
2^{\ell_1 m}
\sum_{\substack{\Nu \in \Lambda_{\ell_1} \times \dots \times \Lambda_{\ell_N}
\\ 
\nu_1 + \dots + \nu_N = \tau}}
\prod_{j=1}^N
\left|
\Box_{\nu_j} F^{j}_{\ell_j, \mu_j}
\right|,
\end{align*}
Recalling that $|\nu_j| \approx 2^{\ell_j}$ if $\nu_j \in \Lambda_{\ell_j}$
and $|\tau| \lesssim 2^{\ell_0 + k}$ if $\tau \in \Lambda_{k, \ell_0}$, 
we have by Proposition \ref{Keyprop}
\begin{align*}
\left\|
\left\|
\sum_{\substack{\Nu \in \Lambda_{\ell_1} \times \dots \times \Lambda_{\ell_N}
\\ 
\nu_1 + \dots + \nu_N = \tau}}
\prod_{j=1}^N
\left|
\Box_{\nu_j} F^{j}_{\ell_j, \mu_j}
\right|
\right\|_{\ell^2_\tau(\Lambda_{k, \ell_0})}
\right\|_{L^{p_0}}
&\lesssim
2^{\min\{\ell_2, \ell_0+k\}n/2}
\prod_{j=3}^N
2^{\ell_jn/2}
\prod_{j=1}^N
2^{-\ell_j\beta(p_j)}
\|F^{j}_{\ell_j, \mu_j}\|_{h^{p_j}}
\\
&\le
2^{\ell_0 n/2}
2^{\min\{\ell_2, k \}n/2}
\prod_{j=3}^N
2^{\ell_jn/2}
\prod_{j=1}^N
2^{-\ell_j\beta(p_j)}
\left\|f_{j, \ell_j}\right\|_{h^{p_j}},
\end{align*}
where we used the fact that the $h^{p_j}$-norms are translation invariant in the last inequality.
Collecting the above estimates, we obtain \eqref{Hulk} with $0 < p < \infty$.

Next we shall show that  the estimate \eqref{Hulk} holds with $p=\infty$.
Notice that $\alpha(\infty) = n/2$.
By the embedding relation $ W^{p_0, 1} \hookrightarrow W^{\infty, 1} \hookrightarrow L^{\infty}$,
we have
\[
\left\|
R_{\bl, k}
\right\|_{L^\infty}
\lesssim
\left\|
R_{\bl, k}
\right\|_{W^{p_0, 1}}
=
\left\|
\left\|
\Box_\tau
R_{\bl, k}
\right\|_{\ell^1_\tau}
\right\|_{L^{p_0}}.
\]
By the same argument as in the case $0< p < \infty$, it holds that
\begin{align*}
\left\|
\left\|
\Box_\tau
R_{\bl, k}
\right\|_{\ell^1_\tau}
\right\|_{L^{p_0}}
&=
\left\|
\left\|
\Box_\tau
\left[
\sum_{\substack{\Nu : \nu_1 + \dots + \nu_N \in \Lambda_{k, \ell_0}
\\
|\nu_1 + \dots + \nu_N - \tau| \lesssim 2^{\ell_0}
}}
Q_{\ell_0, \Nu, \Mu}
\prod_{j=1}^N
\left|
\Box_{\nu_j} F^{j}_{\ell_j, \mu_j}
\right|
\right]
\right\|_{\ell^1_\tau}
\right\|_{L^{p_0}}
\\
&\lesssim
2^{\ell_0 n/p_0}
\sum_{|\lambda| \lesssim 2^{\ell_0}}
\left\|
\left\|
\Box_{\tau-\lambda}
R_\tau
\right\|_{\ell^1_\tau(\Lambda)}
\right\|_{L^{p_0}}
\end{align*}
with $R_\tau$ given by \eqref{Nanjakore}.
Furthermore, it follows from \eqref{pointwiseNikolskij} that
\begin{align*}
\left\|
\left\|
\Box_{\tau-\lambda}
R_\tau
\right\|_{\ell^1_\tau(\Lambda)}
\right\|_{L^{p_0}}
\lesssim
2^{\ell_0 n (1/r_0-1)}
\left\|
\left\|
R_\tau
\right\|_{\ell^1_\tau(\Lambda)}
\right\|_{L^{p_0}}.
\end{align*}
Combining these estimates with Lemma \ref{EST-x} with $L$ replaced by 
$M_0+ n/p_0 +n/r_0 +n$, we obtain 
\begin{align*}
\left\|
R_{\bl, k}
\right\|_{L^\infty}
&\lesssim
2^{-\ell_0(M_0+n)}
\left\|
\left\|
\sum_{\Nu : \nu_1 + \dots + \nu_N  = \tau}
\la \Nu \ra^{m}
\prod_{j=1}^N
\left|
\Box_{\nu_j} F^{j}_{\ell_j, \mu_j}
\right|
\right\|_{\ell^1_\tau(\Lambda)}
\right\|_{L^{p_0}}
\\
&\lesssim
2^{-\ell_0(M_0+n)}
2^{\ell_1 m}
\left\|
\left\|
\sum_{\substack{\Nu \in \Lambda_{\ell_1} \times \dots \times \Lambda_{\ell_N}
\\ 
\nu_1 + \dots + \nu_N = \tau}}
\prod_{j=1}^N
\left|
\Box_{\nu_j} F^{j}_{\ell_j, \mu_j}
\right|
\right\|_{\ell^1_\tau(\Lambda)}
\right\|_{L^{p_0}}.
\end{align*}
Hence, it follows from the estimate in Remark \ref{Keyrem} that
the right hand side just above can be estimated by
\begin{align*}
&
2^{-\ell_0(M_0+n)}
2^{\ell_1 m}
\times
2^{(\ell_0+j)n/2}
2^{\min \{\ell_2, \ell_0+j\} n/2}
\times
2^{-\ell_1\beta(p_1)}
\|f_{1, \ell_1}\|_{h^{p_1}}
\times
2^{-\ell_2\beta(p_2)}
\|f_{2, \ell_2}\|_{h^{p_2}}
\\
&\le
2^{-\ell_0 M_0}
2^{jn/2}
2^{\ell_1 m}
2^{\min \{\ell_2, j\}n/2}
\times
2^{-\ell_1\beta(p_1)}
\|f_{1, \ell_1}\|_{h^{p_1}}
\times
2^{-\ell_2\beta(p_2)}
\|f_{2, \ell_2}\|_{h^{p_2}}.
\end{align*}
The proof of \eqref{Hulk} is complete.

Now, we shall return to the proof the estimate \eqref{GOAL!!!!}.
By the embedding relation \eqref{Bq1Bq2}, 
it is sufficient to prove \eqref{GOAL!!!!} with $0 < q, q_1, \dots, q_N \le \infty$ 
satisfying $1/q = \sum_{j=1}^N 1/q_j$.
We set $r=\min\{1, p, q\}$.

\bigskip
\noindent
\textbf{Estimate for $S_1$ :}
If $\ell_0 \ge \ell_1-3$ and $\ell_1 \ge \ell_j$, $j=2, \dots, N$, then we have
\begin{equation*}
\supp
\F
\left[
Q_{\ell_0, \Nu, \Mu}
\prod_{j=1}^N
\Box_{\nu_j} F^j_{\ell_j, \mu_j}
\right]
\subset
\big\{
 |\zeta| \le 2^{\ell_0+a_N}
\big\}
\end{equation*}
with some positive integer $a_N$ depending only on $N$.
Hence 
\[
\psi_k(D) R_{\bl, k} = 0
\quad 
\text{if}
\quad 
\ell_0 \le k-1-a_N,
\]
and consequently, 
\begin{align} \label{psildel}
\left\|
\psi_k(D) \sum_{\bl \in D_1} R_{\bl, k}
\right\|_{L^{p}}^r
\le
\sum_{\substack{\bl \in D_1 \\ \ell_0 \ge k-a_N}}
\left\|
\psi_k(D)
R_{\bl, k}
\right\|_{L^{p}}^r
\lesssim
\sum_{
\substack{
\bl \in D_1 
\\
\ell_0 \ge k-a_N
}}
\left\|
R_{\bl, k}
\right\|_{h^p}^r,
\end{align}
where we used the estimate \eqref{embd-hpB0pinfty} in $\lesssim$.
Then, by \eqref{Hulk}, the right hand side above is estimated by
\begin{align*}
&
\sum_{
\substack{
\bl \in D_1 
\\
\ell_0 \ge k-a_N
}}
\Big(
2^{-\ell_0 M_0}
2^{k \alpha(p)}
2^{\ell_1 m}
\prod_{j=2}^N
2^{\ell_j n/2}
\prod_{j=1}^N
2^{-\ell_j\beta(p_j)}
\left\| f_{j, \ell_j} \right\|_{h^{p_j}}
\Big)^{r}
=
2^{k \alpha(p) r} U_k,
\end{align*}
where 
\[
U_k
=
\sum_{\ell_0 : \ell_0 \ge k - a_N}
2^{-\ell_0 M_0 r}
\sum_{\ell_1 : \ell_1 \le \ell_0+3}
2^{\ell_1 (m - \beta(p_1)) r}
\|f_{1, \ell_1}\|_{h^{p_1}}^{r}
\prod_{j=2}^N
\sum_{\ell_j : \ell_j \le \ell_1}
2^{\ell_j\theta(p_j) r}
\left\| f_{j, \ell_j} \right\|_{h^{p_j}}^{r}
\]
and
\begin{align}\label{MsMarvel}
\theta(p_j)
=
(n/2)
- \beta(p_j)
=
\max\{n/p_j, n/2\}.
\end{align}
Notice that we have $\left\| f_{j, \ell_j} \right\|_{h^{p_j}} = \|\psi_{\ell_j}(D)f_j\|_{h^{p_j}} \lesssim 2^{-\ell_j s_j} \|f_j\|_{B^{s_j}_{p_j, q_j}}$ 
by \eqref{Bhpequiv} and \eqref{Bq1Bq2}.
Hence we have
\begin{align*}
U_k
&\lesssim
\left(
\sum_{\ell_0 : \ell_0 \ge k - a_N}
2^{-\ell_0 M_0 r}
\sum_{\ell_1 : \ell_1 \le \ell_0+3}
2^{\ell_1(m -\beta(p_1)-s_1) r}
\prod_{j=2}^N
\sum_{\ell_j : \ell_j \le \ell_1}
2^{\ell_j(\theta(p_j) - s_j)r}
\right)
\prod_{j=1}^N
\|f_{j}\|_{B^{s_j}_{p_j, q_j}}^r
\\
&\lesssim
\left(
\sum_{\ell_0 : \ell_0 \ge k-a_N}
2^{-\ell_0(M_0-C) r}
\right)
\prod_{j=1}^N
\|f_{j}\|_{B^{s_j}_{p_j, q_j}}^r.
\end{align*}
Here $C > 0$ is the sufficiently large number satisfying
\[
\sum_{\ell_1 : \ell_1 \le \ell_0+3}
2^{\ell_1(m -\beta(p_1)-s_1) r}
\prod_{j=2}^N
\sum_{\ell_j : \ell_j \le \ell_1}
2^{\ell_j(\theta(p_j) - s_j)r}
\le
2^{\ell_0 C r}.
\]
Choosing $M_0$ sufficiently large, we obtain
\begin{align*}
S_1
&\lesssim
\left\|
2^{k s}
\left\|
\psi_k(D)
\sum_{\bl \in D_1}
R_{\bl, k}
\right\|_{L^p_{}}
\right\|_{\ell^q_k
}
\\
&\lesssim
\left\|
2^{k (s +\alpha(p))}
\left(
\sum_{\ell_0 : \ell_0 \ge k-a_N}
2^{-\ell_0(M_0-C) r}
\right)^{1/r}
\right\|_{\ell^q_k}
\prod_{j=1}^N
\|f_{j}\|_{B^{s_j}_{p_j, q_j}}
\lesssim
\prod_{j=1}^N
\|f_{j}\|_{B^{s_j}_{p_j, q_j}}.
\end{align*}
The estimate for $S_1$ is complete.

\bigskip
\noindent
\textbf{Estimate for $S_2$ :}
If $\ell_0 \le \ell_1-4$ and $\ell_j \le \ell_1-N-2$ for all $j =2, \dots, N$, then
\begin{equation*}
\supp
\F
\left[
Q_{\ell_0, \Nu, \Mu}
\prod_{j=1}^N
\Box_{\nu_j} F^j_{\ell_j, \mu_j}
\right]
\subset
\big\{
2^{\ell_1-b_N} \le |\zeta| \le 2^{\ell_1+b_N}
\big\}.
\end{equation*}
with some positive integer $b_N > 0$ depending only on $N$.
Thus we see that 
\[
\psi_k(D) R_{\bl, k} = 0
\quad
\text{if}
\quad
|\ell_1 - k| \ge b_N + 1.
\]
Hence it follows from the same argument as in \eqref{psildel} that
\begin{align*}
\left\|
\psi_k(D) \sum_{\bl \in D_2} R_{\bl, k}
\right\|_{L^{p}}^r
&\lesssim
\sum_{
\substack{
\bl \in D_{2}
\\
|\ell_1-k| \le b_N
}
}
\left\|
R_{\bl, k}
\right\|_{h^{p}}^r.
\end{align*}
By using \eqref{Hulk} and taking the sum over $\ell_0$,  we have 
\begin{align*}
\sum_{
\substack{
\bl \in D_{2}
\\
|\ell_1-k| \le b_N
}
}
\left\|
R_{\bl, k}
\right\|_{h^{p}}^r
&\lesssim
2^{k \alpha(p) r}
V_k,
\end{align*}
where
\begin{align*}
V_k
&=
\sum_{\ell_1 : |\ell_1-k| \le b_N}
2^{\ell_1(m - \beta(p_1)) r}
\|f_{1, \ell_1}\|_{h^{p_1}}^{r}
\prod_{j=2}^N
\sum_{\ell_j : \ell_j \le \ell_1}
2^{\ell_j\theta(p_j) r}
\left\| f_{j, \ell_j} \right\|_{h^{p_j}}^{r}
\\
&=
\sum_{\ell : |\ell|\le b_N}
2^{(k+\ell)(m-\beta(p_1))r}
\|f_{1, k+\ell}\|_{h^{p_1}}^{r}
\prod_{j=2}^N
\sum_{\ell_j : \ell_j  \le k+\ell}
2^{\ell_j \theta(p_j) r}
\left\| f_{j, \ell_j} \right\|_{h^{p_j}}^{r}.
\end{align*}
Thus, we have
\begin{align*}
S_2
&\lesssim
\left\|
2^{k s}
\left\|
\psi_k(D) \sum_{\bl \in D_2} R_{\bl, k}
\right\|_{L^{p}_{}}
\right\|_{\ell^q_k}
\\
&\lesssim
\sum_{\ell : |\ell|\le b_N}
\left\|
2^{k (s +\alpha(p))}
2^{(k+\ell)(m-\beta(p_1))}
\|f_{1, k+\ell}\|_{h^{p_1}}
\prod_{j=2}^N
\left(
\sum_{\ell_j : \ell_j  \le k+\ell}
2^{\ell_j \theta(p_j) r}
\left\| f_{j, \ell_j} \right\|_{h^{p_j}}^{r}
\right)^{1/r}
\right\|_{\ell^q_k}
\\
&\le
\sum_{\ell : |\ell|\le b_N}
2^{-\ell (s +\alpha(p))}
\left\|
2^{k(m-\beta(p_1) +s +\alpha(p))}
\|f_{1, k}\|_{h^{p_1}}
\prod_{j=2}^N
\left(
\sum_{\ell_j : \ell_j \le k}
2^{\ell_j \theta(p_j) r}
\left\| f_{j, \ell_j} \right\|_{h^{p_j}}^{r}
\right)^{1/r}
\right\|_{\ell^q_k}
\\
&\lesssim
\left\|
2^{k(m-\beta(p_1) +s +\alpha(p))}
\|f_{1, k}\|_{h^{p_1}}
\prod_{j=2}^N
\left(
\sum_{\ell_j : \ell_j \le k}
2^{\ell_j \theta(p_j) r}
\left\| f_{j, \ell_j} \right\|_{h^{p_j}}^{r}
\right)^{1/r}
\right\|_{\ell^q_k}
\end{align*}
where we change a variable with respect to $\ell$ in the third inequality.
Since $\alpha(p) - n/2 =-\min\{n/p, n/2\},$
we have
$m -\beta(p_1) + s + \alpha(p) = s_1 + \sum_{j=2}^N (s_j -\theta(p_j))$.
Hence we obtain
\begin{align*}
S_2
&\lesssim
\left\|
2^{k s_1}
\|f_{1, k}\|_{h^{p_1}}
\prod_{j=2}^N
\left(
\sum_{\ell_j : \ell_j \le k}
2^{(k-\ell_j)(s_j - \theta(p_j)) r}
2^{\ell_j s_j r}
\left\| f_{j, \ell_j} \right\|_{h^{p_j}}^{r}
\right)^{1/r}
\right\|_{\ell^q_k}
\\
&\le
\left\|
2^{k s_1}
\|f_{1, k}\|_{h^{p_1}_{}}
\right\|_{\ell^{q_1}_k}
\prod_{j=2}^N
\left\|
\left(
\sum_{\ell_j : \ell_j \le k}
2^{(k-\ell_j) (s_j -\theta(p_j)) r}
2^{\ell_j s_j r}
\left\| f_{j, \ell_j} \right\|_{h^{p_j}_{}}^{r}
\right)^{1/r}
\right\|_{\ell^{q_j}_k}
\\
&\approx
\|f_{1}\|_{B^{s_1}_{p_1, q_1}}
\prod_{j=2}^N
\left\|
\sum_{\ell_j : \ell_j \le k}
2^{(k-\ell_j) (s_j -\theta(p_j)) r}
2^{\ell_j s_j r}
\left\| f_{j, \ell_j} \right\|_{h^{p_j}}^{r}
\right\|_{\ell^{q_j/r}_k}^{1/r},
\end{align*}
where we used \eqref{Bhpequiv} in the last inequality.
Since $q_j \ge q \ge r$ and $s_j  -\theta(p_j) <0$, $j=2, \dots, N$, it follows from Young's inequality that
\begin{align*}
\left\|
\sum_{\ell_j : \ell_j \le k}
2^{(k-\ell_j) (s_j -\theta(p_j)) r}
2^{\ell_j s_j r}
\left\| f_{j, \ell_j} \right\|_{h^{p_j}}^{r}
\right\|_{\ell^{q_j/r}_k}^{1/r}
&\le
\left\|
2^{k (s_j -\theta(p_j))r}
\right\|_{\ell^1_k}^{1/r}
\left\|
2^{k s_j r}
\left\|f_{j, k}\right\|_{h^{p_j}_{}}^r
\right\|_{\ell^{q_j/r}_k}^{1/r}
\\
&\lesssim
\left\|
2^{k s_j}
\left\|f_{j, k}\right\|_{h^{p_j}_{}}
\right\|_{\ell^{q_j}_k}
\approx
\|f_{j}\|_{B^{s_j}_{p_j, q_j}},
\quad
j=2, \dots, N,
\end{align*}
where we used Proposition \ref{propQui} in the last inequality.
Combining these estimates, we obtain the desired estimate. 

\bigskip
\noindent
\textbf{Estimate for $S_3$ :}
Since if $\ell_0 \le \ell_1-4$ and $\ell_1 \ge \ell_j$, $j =2, \dots, N$, then
\begin{equation*}
\supp
\F
\left[
Q_{\ell_0, \Nu, \Mu}
\prod_{j=1}^N
\Box_{\nu_j} F^j_{\ell_j, \mu_j}
\right]
\subset
\big\{
 |\zeta| \le 2^{\ell_1 + c_N}
\big\}
\end{equation*}
with some positive integer $c_N$ depending only on $N$.
Hence
\[
\psi_k(D) R_{\bl, k} = 0
\quad
\text{if}
\quad 
\ell_1 \le k-1-c_N,
\]
and, by the same argument as above,
\begin{align*}
\left\|
\psi_k(D) \sum_{\bl \in D_3} R_{\bl, k}
\right\|_{L^{p}}^r
&\lesssim
\sum_{
\substack{
\bl \in D_{3}
\\
\ell_1 \ge k-c_N
}
}
\left\|
R_{\bl, k}
\right\|_{h^{p}}^r.
\end{align*}
Since
$
\alpha(p) 
+ 
(n/2)
= 
\max\{n/p^{\prime}, n/2\}
=
\theta(p^{\prime}),
$
it follows from \eqref{Hulk} and \eqref{MsMarvel} that
\begin{align*}
&\left\|
\psi_k(D) \sum_{\bl \in D_3} R_{\bl, k}
\right\|_{L^{p}}^r
\lesssim
2^{k \theta(p^{\prime}) r}
W_k,
\end{align*}
where
\begin{align*}
W_k=
&\sum_{\ell_1 : \ell_1 \ge k-c_N}
2^{\ell_1(m-\beta(p_1)) r}
\|f_{1, \ell_1}\|_{h^{p_1}}^{r}
\\
&\qquad\qquad\qquad
\times
\sum_{\ell_2 : \ell_1-N-1 \le \ell_2 \le \ell_1}
2^{-\ell_2\beta(p_2) r}
\|f_{2, \ell_2}\|_{h^{p_2}}^{r}
\prod_{j=3}^N
\sum_{\ell_j : \ell_j \le \ell_1}
2^{\ell_j\theta(p_j) r}
\left\| f_{j, \ell_j} \right\|_{h^{p_j}}^{r}.
\end{align*}
By a change of variable with respect to $\ell_2$, we have
\begin{align*}
W_k
=
W_{k, 0} + W_{k, 1} + \dots + W_{k, N+1}
=
\sum_{i=0}^{N+1}
W_{k, i}
\end{align*}
with
\begin{align*}
W_{k, i}
=
\sum_{\ell_1 : \ell_1 \ge k-c_N}
2^{\ell_1 (m-\beta(p_1)) r}
2^{-(\ell_1 -i) \beta(p_2) r}
\|f_{1, \ell_1}\|_{h^{p_1}}^{r}
\|f_{2, \ell_1 -i}\|_{h^{p_2}}^{r}
\prod_{j=3}^N
\sum_{\ell_j : \ell_j \le \ell_1}
2^{\ell_j \theta(p_j) r}
\left\| f_{j, \ell_j} \right\|_{h^{p_j}}^{r}.
\end{align*}
Then we have
\begin{align*}
S_3
\lesssim
\left\|
2^{k s}
\left\|
\psi_k(D)
 \sum_{\bl \in D_3} R_{\bl, k}
\right\|_{L^p_{}}
\right\|_{\ell^q_k}
\lesssim
\sum_{i=0}^{N+1}
\left\|
2^{k (s + \theta(p^{\prime}))}
W^{1/r}_{k, i}
\right\|_{\ell^q_k}.
\end{align*}
It is sufficient to prove the estimate for $W_{k,0}$, 
The same argument below works for the other terms.
By using the notation $\theta(p_j)$ given in \eqref{MsMarvel}, we can write
$m -\beta(p_1) -\beta(p_2) = -s -\theta(p^{\prime}) 
+s_1+s_2
+\sum_{j=3}^N (s_j - \theta(p_j) )
$.
Hence
\begin{align*}
W_{k, 0}
&=
\sum_{\ell_1 : \ell_1 \ge k-c_N}
2^{-\ell_1(s+\theta(p^{\prime}))r}
\prod_{j=1}^2
2^{\ell_1 s_j r} \|f_{j, \ell_1}\|_{h^{p_j}}^r
\prod_{j=3}^N
\sum_{\ell_j : \ell_j \le \ell_1}
2^{(\ell_1-\ell_j)(s_j - \theta(p_j)) r}
2^{\ell_j s_j r}
\|f_{j, \ell_j}\|_{h^{p_j}}^{r}
\\
&= 
\sum_{\ell_1 : \ell_1 \ge k-c_N}
2^{-\ell_1(s+\theta(p^{\prime}))r}
\widetilde{W}_{\ell_1}, 
\end{align*}
where
\begin{align*}
\widetilde{W}_{\ell_1}
=
\prod_{j=1}^2
2^{\ell_1 s_j r} \|f_{j, \ell_1}\|_{h^{p_j}}^r
\prod_{j=3}^N
\sum_{\ell_j : \ell_j \le \ell_1}
2^{(\ell_1-\ell_j)(s_j - \theta(p_j)) r}
2^{\ell_j s_j r}
\|f_{j, \ell_j}\|_{h^{p_j}}^{r}.
\end{align*}
Since $q \ge r$, it follows from Young's inequality that
\begin{align*}
\left\|
2^{k (s +\theta(p^{\prime}))}
W_{k, 0}^{1/r}
\right\|_{\ell^q_k}
&=
\left\|
\sum_{\ell_1 : \ell_1 \ge k-c_N}
2^{-(\ell_1-k)(s+\theta(p^{\prime}))r}
\widetilde{W}_{\ell_1}
\right\|_{\ell^{q/r}_k}^{1/r}
\\
&\lesssim
\left\|
\sum_{\ell_1 \in  \N_0}
2^{-|\ell_1-k|(s+\theta(p^{\prime}))r}
\widetilde{W}_{\ell_1}
\right\|_{\ell^{q/r}_k}^{1/r}
\\
&\le
\left\|
2^{-k(s+\theta(p^{\prime}))r}
\right\|_{\ell^1_k}^{1/r}
\|
\widetilde{W}_{k}
\|_{\ell^{q/r}_k}^{1/r}
\lesssim
\|
\widetilde{W}_{k}
\|_{\ell^{q/r}_k}^{1/r},
\end{align*}
where we used the assumption $s+ \theta(p^{\prime}) > 0$ in the last inequality. 
Furthermore, 
by H\"older's inequality and  the same argumet as in  the estimate for $S_2$, we obtain
\begin{align*}
\|
\widetilde{W}_{k}
\|_{\ell^{q/r}_k}^{1/r}
&\lesssim
\prod_{j=1}^2
\left\|
2^{k s_j r} \left\|f_{j, k}\right\|_{h^{p_j}_{}}^r
\right\|_{\ell^{q_j/r}_k}^{1/r}
\prod_{j=3}^N
\left\|
\sum_{\ell_j : \ell_j \le k}
2^{(k-\ell_j)(s_j - \theta(p_j)) r}
2^{\ell_j s_j r}
\left\| f_{j, \ell_j} \right\|_{h^{p_j}_{}}^{r}
\right\|_{\ell^{q_j/r}_k}^{1/r}
\\
&\lesssim
\prod_{j=1}^N
\|f_j\|_{B^{s_j}_{p_j, q_j}}.
\end{align*}

\bigskip
Thus we obtain the estimate \eqref{GOAL!!!!}. The proof of Theorem \ref{main1} is complete.

\section{Proof of Theorem \ref{thmnec}}

In this section, we shall show the sharpness of conditions of $s_1, \dots, s_N$ and $s$.
In particular, we shall give the proof of Theorem \ref{thmnec}.

We now assume that the boundedness
\begin{equation} \label{nec-bdd}
\Op(S^m_{0,0}(\R^n, N))
\subset 
B(B^{s_1}_{p_1, q_1} \times \dots \times B^{s_N}_{p_N, q_N} \to B^s_{p, q})
\end{equation}
holds with 
\begin{equation}\label{mcrit}
m = \min \left\{ \frac{n}{p}, \frac{n}{2} \right\} 
-
\sum_{j=1}^N \max\left\{ \frac{n}{p_j}, \frac{n}{2} \right\}
+
\sum_{j=1}^N s_j  -s.
\end{equation}
By the closed graph theorem, the assumption \eqref{nec-bdd} implies that there exists
$M \in \N$ such that
\begin{align}\label{operator-norm}
\begin{split}
&\|T_{\sigma}\|_{B^{s_1}_{p_1, q_1} \times \dots \times B^{s_N}_{p_N, q_N} \to B^{s}_{p, q}}
\\
&\quad
\lesssim
\max_{|\alpha|, |\beta_1|, \dots, |\beta_N| \le M}
\|
\la (\xi_1, \dots, \xi_N)\ra^{-m}
\pa_x^{\alpha}\pa_{\xi_1}^{\beta_1} \dots \pa_{\xi_N}^{\beta_N}
\sigma(x, \xi_1, \dots, \xi_N)
\|_{L^\infty_{x, \xi_1, \dots, \xi_N}}
\end{split}
\end{align}
for all $\sigma \in S^{m}_{0, 0}(\R^n, N)$. 
For the argument using the closed graph theorem, see \cite[Lemma 2.6]{BBMNT}.

We recall the following fact given 
by Wainger \cite{Wainger} and Miyachi-Tomita \cite{MT-IUMJ}.
\begin{lem}[\cite{Wainger, MT-IUMJ}]\label{lemWainger}
Let $0< a < 1$, $0< b < n$, $1 \le p \le \infty$ and $\varphi \in \Sh(\R^n)$.  
For $\epsilon > 0$, we set
\[
f_{a, b, \epsilon}(x)
=
\sum_{\nu \in \Z^n \setminus \{0\}}
e^{-\epsilon|\nu|}
|\nu|^{-b} e^{i|\nu|^a} e^{i\nu \cdot x} \varphi(x).
\]
If $b > (1-a)(n/2-n/p) +n/2$, then $\sup_{\epsilon >0}\|f_{a, b, \epsilon}\|_{L^p} < \infty$.
\end{lem}

In this section, we will use the following partition of unity $\psi_\ell$, $\ell= 0, 1, 2, \dots$, 
in the definition of Besov spaces;
\begin{align*}
&\supp \psi_0 \subset \{|\xi| \le 2^{3/4}\}, 
\quad
\psi_0 = 1 \quad \text{on} \quad \{|\xi| \le 2^{1/4}\}
\\
&\supp \psi_\ell \subset \{2^{\ell-3/4} \le |\xi| \le 2^{\ell+3/4}\}, 
\quad
\psi_\ell= 1 \quad \text{on} \quad \{2^{\ell -1/4} \le |\xi| \le 2^{\ell + 1/4}\},
\quad 
\ell \ge 1,
\\
&\|\pa^{\alpha}\psi_\ell\|_{L^\infty}
\le
C_{\alpha} 2^{-\ell |\alpha|},
\quad \alpha \in \N_0^n, \xi \in \R^n,
\\
&\sum_{\ell = 0}^\infty \psi_{\ell}(\xi) = 1, 
\quad \xi \in \R^n.
\end{align*}

We take functions $\widetilde{\psi}_\ell \in \Sh(\R^n)$, $\ell = 0, 1, 2, \dots$, such that 
\begin{align*}
&\supp \widetilde{\psi}_0 
\subset 
\{|\xi| \le 2^{1/4}\}
\quad
\text{and}
\quad
\widetilde{\psi}_0
=
1
\quad
\text{on}
\quad
\{|\xi| \le 2^{1/8}\},
\\
&\supp \widetilde{\psi}_\ell \subset \{2^{\ell-1/4} \le |\xi| \le 2^{\ell+1/4}\}, 
\quad
\widetilde{\psi}_\ell = 1 \quad \text{on} \quad \{2^{\ell -1/8} \le |\xi| \le 2^{\ell + 1/8}\},
\quad 
\ell \ge 1,
\\
&\sup_{\ell \in \N_0}\|\F^{-1} \widetilde{\psi}_{\ell}\|_{L^1} < \infty.
\end{align*}

We also use the functions $\vphi, \widetilde{\vphi} \in \Sh(\R^n)$ satisfying the following;
\begin{align*}
&\supp \vphi \subset [-1/4, 1/4]^n,
\quad
|\F^{-1}\vphi(x)| \ge 1 \quad \text{on} \quad [-1, 1]^n,
\\
&
\supp \widetilde{\vphi} \subset [-1/2, 1/2]^n,
\quad
\widetilde{\vphi} = 1 \quad \text{on} \quad  [-1/4, 1/4]^n.
\end{align*}

\begin{lem}\label{lemNecessity1}
Let $1< p_1, \dots,  p_N < \infty$, $0< p< \infty$, $0 < q, q_1, \dots, q_N < \infty$ and $s, s_1, \dots, s_N \in \R$. 
If  the boundedness \eqref{nec-bdd} holds with $m$ given in \eqref{mcrit}, then 
$s \ge -\max\{n/p^{\prime}, n/2\}$.
\end{lem}

\begin{proof}
Let $\{c_{\Mu}\}_{\Mu \in (\Z^n)^{N}}$ be a sequence of complex numbers satisfying $\sup_{\Mu} |c_{\Mu}| \le 1$.
We take sufficiently large number $L > 0$. 
For sufficiently large $\ell \in \N$ satisfying $\ell > L$,
we set 
\begin{align*}
&\sigma_{\ell}(\xi_1, \dots, \xi_N)
=
\sum_{\Mu = (\mu_1,\dots, \mu_N) \in D_\ell}
c_{\Mu} \la \Mu \ra^{m}
\prod_{j=1}^N
\vphi(\xi_j -  \mu_j),
\\
&
\widehat{f_{1, \ell}}(\xi_1)
=
\widetilde{\psi}_{\ell}(\xi_1) \widehat{f_{a_1, b_1, \epsilon}}(\xi_1),
\quad
\widehat{f_{2, \ell}}(\xi_2)
=
\widetilde{\psi}_{\ell}(\xi_2) \widehat{f_{a_2, b_2, \epsilon}}(\xi_2),
\\
&
\widehat{f_{j, \ell}}(\xi_j)
=
\widetilde{\psi}_{\ell- L}(\xi_j) \widehat{f_{a_j, b_j, \epsilon}}(\xi_j),
\quad
j=3, \dots, N.
\end{align*}
where
\begin{align*}
&D_{\ell} 
= 
\left\{
\Mu = (\mu_1, \dots, \mu_N) \in (\Z^n)^N 
: 
\begin{array}{l}
\mu_1+ \mu_2 + \dots + \mu_N = 0,
\quad
2^{\ell -\delta} \le |\mu_2| \le 2^{\ell  + \delta},
\\
2^{\ell -L-\delta} \le |\mu_j| \le 2^{\ell -L + \delta}, 
\quad
j=3, \dots, N
\end{array}
\right\}
\end{align*}
with sufficiently small $\delta >0$ and
\[
0< a_j < 1,
\quad
b_j
=
(1-a_j)
\left(
n/2-n/p_j
\right) 
+
n/2
+ 
\epsilon_j,
\quad
\epsilon_j > 0,
\quad
j=1, \dots, N.
\]
We choose the number $L$ sufficiently large enough to satisfy
\begin{equation}\label{setassum1}
\Mu \in D_{\ell}
\implies
2^{\ell -2 \delta}
\le
|\mu_2 + \dots + \mu_N|
\le
2^{\ell +2\delta}.
\end{equation}
Then, since $(1+|\xi_1|+\dots +|\xi_N|) \approx (1+|\mu_1|+\dots +|\mu_N|)$
if $(\xi_1, \dots, \xi_N) \in \supp \sigma_\ell$  and since $\sup_{\Mu}  |c_{\Mu}| \le 1$, 
we have $\sigma_{\ell} \in S^m_{0,0}(\R^n, N)$.
Furthermore, since 
$\psi_{k} \widetilde{\psi}_{\ell} = \widetilde{\psi}_{\ell}$ 
if $k = \ell$ 
and 
$0$ 
otherwise, 
Lemma \ref{lemWainger} and Young's inequality yield that
\begin{align}\label{f1}
\begin{split}
\|f_{j, \ell}\|_{B^{s_j}_{p_j, q_j}} 
&=2^{\ell s_j}\|\widetilde{\psi}_{\ell}(D)f_{a_j, b_j, \epsilon}\|_{L^{p_j}}
\\
&\le
2^{\ell s_j}\|\F^{-1}\widetilde{\psi}_{\ell}\|_{L^1} \|f_{a_j, b_j, \epsilon}\|_{L^{p_j}}
\lesssim 
2^{\ell s_j},
\quad
j=1, 2,
\end{split}
\end{align}
and similarly, 
\begin{align}\label{f3N}
\|f_{j, \ell}\|_{B^{s_j}_{p_j, q_j}} 
\lesssim 
2^{\ell s_j},
\quad
j=3, \dots, N.
\end{align}
Here the implicit constants are independent of $\ell$ and $\epsilon$.

Now, since $\ell$ is sufficiently large and $\delta$ is small, we have
\begin{align*}
&
\supp \vphi (\cdot - \mu_j)
\subset
\{2^{\ell-1/8} \le |\xi_j| \le 2^{\ell + 1/8}\},
\quad
j=1, 2,
\\
&
\supp \vphi (\cdot - \mu_j)
\subset
\{2^{\ell-L-1/8} \le |\xi_j| \le 2^{\ell-L + 1/8}\},
\quad
j=3, \dots, N,
\end{align*}
for $\Mu = (\mu_1, \dots, \mu_N) \in D_{\ell}$,
and consequently we have
\begin{align*}
&
\varphi(\xi_j - \mu_j)
\widetilde{\psi}_{\ell}(\xi_j)
=
\varphi(\xi_j - \mu_j),
\quad
j=1,2,
\\
&
\varphi(\xi_j - \mu_j)
\widetilde{\psi}_{\ell-L}(\xi_j)
= 
\varphi(\xi_j - \mu_j),
\quad
j=3, \dots, N.
\end{align*}
Hence, we obtain
\begin{align*}
T_{\sigma_\ell}(f_{1, \ell}, \dots,  f_{N, \ell})(x)
&=
\sum_{\Mu \in D_{\ell}}
c_{\Mu} \la \Mu \ra^{m}
\prod_{j=1}^N
e^{-\epsilon|\mu_j|}
|\mu_j|^{-b_j}
e^{i|\mu_j|^{a_j}}
\F^{-1}[\vphi(\cdot - \mu_j)](x)
\\
&=
\{\Phi(x)\}^N
\sum_{\Mu \in D_{\ell}}
c_{\Mu} \la \Mu \ra^{m}
\prod_{j=1}^N
e^{-\epsilon|\mu_j|}
|\mu_j|^{-b_j}
e^{i|\mu_j|^{a_j}}
\end{align*}
with $\Phi = \F^{-1} \varphi$,
where we used the fact that $\mu_1 + \dots + \mu_N = 0$ for $\Mu = (\mu_1, \dots, \mu_N) \in D_{\ell}$.
Thus, if we choose $c_{\Mu} = \prod_{j=1}^N e^{-i|\mu_j|^{a_j}}$, then
\[
T_{\sigma_\ell}(f_{1, \ell}, \dots, f_{N, \ell})(x)
=
\{\Phi(x)\}^N
\sum_{\Mu \in D_{\ell}}
\la \Mu \ra^{m}
\prod_{j=1}^N
e^{-\epsilon|\mu_j|}
|\mu_j|^{-b_j}.
\]
Hence we obtain
\begin{equation}\label{normest33}
\|T_{\sigma_\ell}(f_{1, \ell}, \dots, f_{N, \ell})\|_{B^s_{p, q}} 
\approx 
\sum_{\Mu \in D_{\ell}}
\la \Mu \ra^{m}
\prod_{j=1}^N
e^{-\epsilon|\mu_j|}
|\mu_j|^{-b_j}.
\end{equation}
Combining \eqref{nec-bdd}, \eqref{f1}, \eqref{f3N},  and \eqref{normest33}, we obtain
\begin{align*}
\sum_{\Mu \in D_{\ell}}
\la \Mu \ra^{m}
\prod_{j=1}^N
e^{-\epsilon|\mu_j|}
|\mu_j|^{-b_j}
\lesssim
2^{\ell \sum_{j=1}^N s_j}
\end{align*}
where the implicit constant does not depend on $\epsilon$.
Hence, taking the limit $\epsilon \to 0$, we have
\begin{align*}
\sum_{\Mu \in D_{\ell}}
\la \Mu \ra^{m}
\prod_{j=1}^N
|\mu_j|^{-b_j}
\lesssim
2^{\ell \sum_{j=1}^N s_j}.
\end{align*}
By \eqref{setassum1}, the left hand side above can be written as 
\begin{align*}
&
\sum_{\Mu \in D_{\ell}}
\la \Mu \ra^{m}
\prod_{j=1}^N
|\mu_j|^{-b_j}
\\
&=
\sum_{\substack{
2^{\ell-\delta} \le |\mu_2| \le 2^{\ell+\delta} 
\\
2^{\ell-L-\delta} \le |\mu_j| \le 2^{\ell-L+\delta},
\ 
j=3, \dots, N
}}
\la (\mu_2 + \dots + \mu_N, \mu_2, \dots, \mu_N) \ra^m
|\mu_2+ \dots + \mu_N|^{-b_1}
\prod_{j=2}^N
|\mu_j|^{-b_j}
\\
&\approx
2^{\ell(m-\sum_{j=1}^N b_j)}
\\
&\qquad\qquad
\times
\mathrm{card}
\left\{
(\mu_2, \dots, \mu_N) \in (\Z^n)^{N-1}
:
\begin{array}{l}
2^{\ell-\delta} \le |\mu_2| \le 2^{\ell+\delta},
\\
2^{\ell-L-\delta} \le |\mu_j| \le 2^{\ell-L+\delta},
\ 
j=3, \dots, N
\end{array}
\right\}
\\
&\approx
2^{\ell(m-\sum_{j=1}^N b_j+(N-1)n)}.
\end{align*}
Thus we obtain
\begin{equation*}
2^{\ell(m-\sum_{j=1}^N b_j+(N-1)n)}
\lesssim
2^{\ell \sum_{j=1}^N s_j}
\end{equation*}
with the implicit constant independent of $\ell$.
Since $\ell$ is arbitrarily large, we obtain
\begin{align*}
\sum_{j=1}^N 
s_j
&\ge 
m-\sum_{j=1}^N b_j +(N-1)n
=
m
-\sum_{j=1}^N 
\left\{
(1-a_j)\left(\frac{n}{2}-\frac{n}{p_j}\right) +\frac{n}{2}
+ \epsilon_j
\right\} +(N-1)n.
\end{align*}
Taking the limits as $a_j \to 0$ if $1 < p_j \le 2$, $a_j \to 1$ if $ 2 < p_j < \infty$, and $\epsilon_j \to 0$, we conclude that
\begin{align*}
\sum_{j=1}^N 
s_j 
\ge 
m
-
n
+
\sum_{j =1}^N
\max
\left\{
\frac{n}{p_j}, \frac{n}{2}
\right\}
=
-n
+\min
\left\{
\frac{n}{p}, \frac{n}{2}
\right\}
+
\sum_{j=1}^N 
s_j 
-
s
,
\end{align*}
which means 
$s \ge -\max \{n/p^{\prime}, n/2 \}$.
The proof is complete.
\end{proof}

\begin{lem}\label{lemNecessity2}
Let $1< p_1, \dots,  p_N < \infty$, $0< p< \infty$, $1 < q, q_1, \dots, q_N < \infty$ and $s, s_1, \dots, s_N \in \R$. 
If  the boundedness \eqref{nec-bdd} holds with $m$ given in \eqref{mcrit}, then $s_j \le \max \{n/p_j, n/2\}$, $j=1, \dots, N$.
\end{lem}

\begin{proof}
It is sufficient to prove that $s_1 \le \max\{n/p_1, n/2\}$ by symmetry.

\bigskip
\textbf{Case I : $2 < p < \infty$.}
It is well known that 
the multilinear pseudo-differential operator $T_{\sigma}$ with $\sigma \in S^m_{0, 0}(\R^n, N)$ is bounded from 
$B^{s_1}_{p_1, q_1}  
\times 
B^{s_2}_{p_2, q_2} 
\times  
\dots
\times B^{s_N}_{p_N, q_N}$ to $B^{s}_{p, q}$, then
$T_{\sigma^{*1}}$ is bounded from 
$B^{-s}_{p^\prime, q^\prime} \times \dots \times B^{s_N}_{p_N, q_N} \times $ 
to 
$B^{-s_1}_{p_1^\prime, q_1^\prime}$. 
Here $\sigma^{*1}$ is a multilinear symbol defined by
\[
\int_{\R^n}
T_{\sigma}(f_1, f_2, \dots, f_N)(x) g(x)
\,
dx
=
\int_{\R^n}
T_{\sigma^{*1}}(g, f_2, \dots, f_N)(x) f_1(x)
\,
dx.
\]
Since $\sigma \in S^m_{0, 0}(\R^n, N)$ if and only if $\sigma^{*1} \in S^{m}_{0,0}(\R^n, N)$,
the boundedness \eqref{nec-bdd} holds if and only if
\[
\Op(S^m_{0, 0}(\R^n, N))
\subset
B(B^{-s}_{p^\prime, q^\prime} \times \dots \times B^{s_N}_{p_N, q_N} \to B^{-s_1}_{p_1^\prime, q_1^\prime}).
\]
Thus it follows from Lemma \ref{lemNecessity1} that
$-s_1 \ge -\max\{n/(p_1^\prime)^{\prime}, n/2 \}$,
that is, 
$s_1 \le \max\{n/p_1, n/2\}$.

\bigskip
\textbf{Case I\hspace{-0.1em}I : $0< p \le 2$.}
Let $L \in \N_0$ be a sufficently large.
For $\ell \in \N$ satisfying $\ell > L$, we set 
\begin{align*}
&\sigma_{\ell}(\xi_1, \dots, \xi_N)
=
\vphi(\xi_1)
\left(
\sum_{\Mu = (\mu_2, \dots, \mu_N)  \in D_{\ell}}
c_{\Mu} \la \Mu \ra^{m}
\prod_{j=2}^{N}
\vphi(\xi_j-\mu_j)
\right)
,
\\
&
\widehat{f}_1(\xi_1) = \widetilde{\vphi}(\xi_1),
\quad
\widehat{f_{2, \ell}}(\xi_2) 
=
\widetilde{\psi}_{\ell}(\xi_2) \widehat{f_{a_2, b_2, \epsilon}}(\xi_{2}),
\\
&\widehat{f_{j, \ell}}(\xi_j) 
= 
\widetilde{\psi}_{\ell-L}(\xi_j) \widehat{f_{a_j, b_j, \epsilon}}(\xi_j),
\quad
j=3, \dots, N,
\end{align*}
where $\{c_{\Mu}\}_{\Mu}$ be a sequence of complex numbers such that $\sup_{\Mu} |c_{\Mu}| \le 1$, and
\begin{align*}
D_{\ell}
=
\left\{
\Mu
=
(\mu_2, \dots, \mu_{N})
\in (\Z^n)^{N-1}
:
\begin{array}{l}
2^{\ell - \delta}
\le
|\mu_2 + \dots + \mu_{N}|
\le
2^{\ell + \delta},
\\
2^{\ell - L - \delta}
\le
|\mu_j|
\le
2^{\ell - L + \delta},
\quad
j=3, \dots, N
\end{array}
\right\}.
\end{align*}
with sufficiently small $\delta > 0$.
We take the number $L$ sufficiently large so that 
\begin{align}\label{murange2}
\Mu
\in 
D_{\ell}
\implies
2^{\ell - 2\delta}
\le
|\mu_{2}|
\le
2^{\ell + 2\delta}.
\end{align}
We see that $\sigma_{\ell} \in S^m_{0,0}(\R^n, N)$ by the same argument as in the proof of Lemma \ref{lemNecessity1}, and hence \eqref{operator-norm} yields that 
\begin{equation}\label{operator-norm-2}
\|T_{\sigma_\ell}\|_{B^{s_1}_{p_1, q_1} \times \dots \times B^{s_N}_{p_N, q_N} \to B^{s}_{p, q}} \lesssim
1
\end{equation}
with the implicit constant independent of $c_{\Mu}$ and $\ell$.
We also have
\begin{align}\label{f2}
\|f_1\|_{B^{s_1}_{p_1,q_1}} \lesssim 1,
\quad
\|f_{j, \ell}\|_{B^{s_j}_{p_j, q_j}} \lesssim 2^{\ell s_j}, \quad j=2, \dots, N.
\end{align}
Moreover, since $\vphi\widetilde{\vphi} = \vphi$ and
\begin{align*}
&\vphi(\xi_2 - \mu_2) \widetilde{\psi}_{\ell}(\xi_2)  = \vphi(\xi_2 - \mu_2)
\\
&\vphi(\xi_j-\mu_j)\widetilde{\psi}_{\ell-L}(\xi_j) = \vphi(\xi_j-\mu_j), \quad j=3, \dots, N,
\end{align*}
if $\Mu  \in D_{\ell}$, then we have
\begin{align*}
T_{\sigma_{\ell}}(f_{1}, f_{2, \ell} \dots, f_{N, \ell})(x)
&=
\{\Phi(x)\}^N
\sum_{\Mu \in D_{\ell}}
c_{\Mu} \la \Mu \ra^{m}
\prod_{j=2}^{N}
e^{-\epsilon |\mu_j|}
|\mu_j|^{-b_j}
e^{i|\mu_j|^{a_j}}
e^{i \mu_j \cdot x}
\end{align*}
Let $\{r_{\mu}(\omega)\}_{\mu \in \Z^n}$, $\omega \in [0, 1]^n$, be a sequence of the Rademacher functions 
enumerated in such a way that their index set is $\Z^n$
(for the definition of the Rademacher function, see, e.g., \cite[Appendix C]{Grafakos-Classical}). 
If we choose $\{c_{\Mu}\}_{\Mu}$ as 
\[
c_{\Mu} = r_{\mu_2+ \dots + \mu_{N}}(\omega) \prod_{j=2}^{N} e^{-i |\mu_j|^{a_j}},
\]
then
\begin{align*}
T_{\sigma_{\ell}}(f_{1}, f_{2, \ell} \dots, f_{N, \ell})(x)
&=
\{\Phi(x)\}^N
\sum_{\Mu \in D_{\ell}}
r_{\mu_2 + \dots + \mu_N}(\omega) 
\la \Mu \ra^{m}
\prod_{j=2}^{N}
e^{-\epsilon |\mu_j|}
|\mu_j|^{-b_j}
e^{i \mu_j \cdot x}
\\
&=
\{\Phi(x)\}^N
\sum_{\nu : 2^{\ell-\delta} \le |\nu| \le 2^{\ell + \delta}}
r_{\nu}(\omega) 
e^{i \nu \cdot x}
\sum_{\substack{\Mu \in D_{\ell} \\ \mu_2 + \dots + \mu_{N} = \nu}}
\la \Mu \ra^{m}
\prod_{j=2}^{N}
e^{-\epsilon |\mu_j|}
|\mu_j|^{-b_j}.
\end{align*}
Since  $\ell$ is sufficiently large, we have
\begin{align*}
\supp \F[T_{\sigma_{\ell}}(f_1, f_{2, \ell}, \dots, f_{N, \ell})]
&
\subset
\bigcup_{\nu : 2^{\ell-\delta} \le |\nu| \le 2^{\ell + \delta}}
\supp[\vphi * \dots * \vphi](\cdot - \nu)
\subset
\{
2^{\ell-1/4}
\le
|\zeta|
\le
2^{\ell+1/4}
\},
\end{align*}
and consequently, we obtain 
$\|T_{\sigma_{\ell}}(f_1, f_{2, \ell}, \dots, f_{N, \ell})\|_{B^{s}_{p,q}}= 2^{\ell s}\|T_{\sigma_{\ell}}(f_1, f_{2, \ell}, \dots, f_{N, \ell})\|_{L^p}$.
By the assumption $|\Phi| \ge 1$ on $[-1, 1]^n$, 
\begin{align*}
\int_{[0, 1]^n}
\|T_{\sigma_{\ell}}(f_1, f_{2, \ell}, \dots, f_{N, \ell})\|_{B^{s}_{p,q}}^p
\,
d\omega
&=
2^{\ell s p}
\int_{[0, 1]^n}
\|T_{\sigma_{\ell}}(f_1, f_{2, \ell}, \dots, f_{N, \ell})\|_{L^p}^p
\,
d\omega
\\
&\gtrsim
2^{\ell s p}
\int_{[-1, 1]^n}
\left(
\int_{[0, 1]^n}
\left|
\sum_{\nu : 2^{\ell-\delta} \le |\nu| \le 2^{\ell + \delta}}
r_{\nu}(\omega) 
e^{i \nu \cdot x}
d_{\nu, \epsilon}
\right|^p
\,
d\omega
\right)
dx,
\end{align*}
where $d_{\nu, \epsilon}$ is defined by
\[
d_{\nu, \epsilon}
=
\sum_{\substack{\Mu \in D_{\ell} \\ \mu_2 + \dots + \mu_{N} =\nu} }
\la \Mu \ra^{m}
\prod_{j=2}^{N}
e^{-\epsilon |\mu_j|}
|\mu_j|^{-b_j}.
\]
Hence, by Khintchine's inequality (see, e.g., \cite[Appendix C]{Grafakos-Classical}), 
we have
\begin{align*}
\int_{[-1, 1]^n}
\left(
\int_{[0, 1]^n}
\left|
\sum_{\nu : 2^{\ell-\delta} \le |\nu| \le 2^{\ell + \delta}}
r_{\nu}(\omega) 
e^{i \nu \cdot x}
d_{\nu, \epsilon}
\right|^p
\,
d\omega
\right)
dx
&\approx
\int_{[-1, 1]^n}
\left(
\sum_{\nu : 2^{\ell-\delta} \le |\nu| \le 2^{\ell + \delta}}
|d_{\nu, \epsilon}|^{2}
\right)^{p/2}
dx
\\
&\approx
\left(
\sum_{\nu : 2^{\ell-\delta} \le |\nu| \le 2^{\ell + \delta}}
|d_{\nu, \epsilon}|^{2}
\right)^{p/2}
\end{align*}
Combining 
\eqref{nec-bdd}, 
\eqref{operator-norm-2}
and
\eqref{f2},
we obtain
\begin{align*}
\left(
\sum_{\nu : 2^{\ell-\delta} \le |\nu| \le 2^{\ell + \delta}}
|d_{\nu, \epsilon}|^{2}
\right)^{1/2}
&\lesssim
2^{-\ell s}
\times
\left(
\int_{[0, 1]^n}
\|T_{\sigma_{\ell}}(f_1, f_{2, \ell}, \dots, f_{N, \ell})\|_{B^{s}_{p,q}}^p
\,
d\omega
\right)^{1/p}
\\
&\lesssim
2^{-\ell s}
\times
\|f_1\|_{B^{s_1}_{p_1, q_1}}
\prod_{j=2}^N
\|f_{j, \ell}\|_{B^{s_j}_{p_j, q_j}}
\\
&\lesssim
2^{\ell (\sum_{j=2}^{N} s_j -s)}
\end{align*}
Since the sums over $\nu$ and $\Mu$ are finite sums,  
we have by taking the limit $\epsilon \to 0$
\begin{equation} \label{estfordnu}
\left(
\sum_{\nu : 2^{\ell-\delta} \le |\nu| \le 2^{\ell + \delta}}
|d_{\nu}|^{2}
\right)^{1/2}
\lesssim
2^{\ell (\sum_{j=2}^{N} s_j -s)}
\end{equation}
with 
\[
d_{\nu}
=
\sum_{\substack{\Mu \in D_{\ell} \\ \mu_2 + \dots + \mu_{N} =\nu} }
\la \Mu \ra^{m}
\prod_{j=2}^{N}
|\mu_j|^{-b_j}.
\]
Recalling \eqref{murange2}, 
for $\nu \in \Z^n$ satisfying $\nu : 2^{\ell-\delta} \le |\nu| \le 2^{\ell + \delta}$, we have 
\begin{align*}
d_{\nu}
&\approx
2^{\ell (m-\sum_{j=2}^{N} b_j)}
\\
&\qquad
\times
\mathrm{card}
\left\{
(\mu_3, \dots, \mu_N) \in (\Z^n)^{N-2} : 
2^{\ell - L - \delta}
\le
|\mu_j|
\le
2^{\ell - L + \delta},
\ 
j=3, \dots, N
\right\}
\\
&\approx
2^{\ell (m-\sum_{j=2}^{N} b_j-(N-2)n)}.
\end{align*}
Combining the above estimates, we obtain
\[
2^{\ell(m -\sum_{j=2}^{N} b_j +(N-2)n + n/2)}
\lesssim
2^{\ell (\sum_{j=2}^{N} s_j -s)}
\]
with the implicit constant independent of $\ell$.
Since $\ell$ is sufficiently large, we obtain
\begin{align*}
\sum_{j=2}^{N}
s_j 
-s
&\ge
m
-\sum_{j=2}^{N} b_j
+(N-2)n + \frac{n}{2}
\\
&=
m
-
\left\{
\sum_{j=2}^{N}
(1-a_j) 
\left(
\frac{n}{2}-\frac{n}{p_j}
\right)
+\frac{n}{2}
+
\epsilon_j
\right\}
+(N-2)n + \frac{n}{2}
\end{align*}
If we take limits as $a_j \to 0$ if $1 < p_j \le 2$, 
$a_j \to 1$ if $2 < p_j < \infty$ and $\epsilon_j \to 0$,
then we obtain
\begin{align*}
\sum_{j=2}^{N}
s_j 
-s
\ge
m
+
\sum_{j=2}^{N}
\max\left\{ \frac{n}{p_j}, \frac{n}{2}\right\}
- \frac{n}{2}
=
-\max\left\{\frac{n}{p_1}, \frac{n}{2}\right\}
+
\sum_{j=1}^N
s_j
-s,
\end{align*}
which means
$s_1 \le \max\{n/p_1, n/2\}$.
The proof  is complete.
\end{proof}

In the rest of this section, we prove Theorem \ref{thmnec}. 
The basic ideas  go back to \cite{KMT-JFA} and \cite{Shida-Sobolev}.
\begin{proof}[Proof of Theorem \ref{thmnec}]
Let $0 < p, p_1, \dots, p_N \le \infty$, $0 < q, q_1, \dots, q_N \le \infty$ and $s, s_1 \dots, s_N \in \R$.
We assume that the boundedness
\begin{align*}
&
\Op(S^m_{0, 0}(\R^n, N))
\subset
B(B^{s_1}_{p_1, q_1} \times \dots \times B^{s_N}_{p_N, q_N} \to B^{s}_{p, q})
\end{align*}
holds with
\begin{align*}
m = \min \left\{ \frac{n}{p}, \frac{n}{2} \right\} 
-
\sum_{j=1}^N \max\left\{ \frac{n}{p_j}, \frac{n}{2} \right\}
+
\sum_{j=1}^N s_j  -s.
\end{align*}
It is already proved in Theorem \ref{main1} that the boundedness
\[
\Op(S^{-(N-1)(n/2-t)}_{0, 0}(\R^n, N))
\subset
B(B^{t}_{2, 2} \times \dots \times B^{t}_{2, 2} \to B^{t}_{2, 2})
\]
holds for any $t \in \R$ satisfying $-n/2 < t < n/2$.
Then, by complex interpolation, we have
\[
\Op(S^{\widetilde{m}}_{0, 0}(\R^n, N))
\subset
B(B^{\widetilde{s}_1}_{\widetilde{p}_1, \widetilde{q}_1} 
\times \dots \times 
B^{\widetilde{s}_N}_{\widetilde{p}_N, \widetilde{q}_N} \to B^{\widetilde{s}}_{\widetilde{p}, \widetilde{q}}),
\]
where $0< \theta < 1$,  
\begin{align*}
&1/\widetilde{p}_j
=
(1-\theta)/2
+
\theta/p_j,
\quad
1/\widetilde{q}_j
=
(1-\theta)/2
+
\theta/q_j,
\quad
\widetilde{s_j} 
=
(1-\theta)t + \theta s_j ,
\quad
j=1, \dots, N,
\\
&1/\widetilde{p}
=
(1-\theta)/2
+
\theta/p,
\quad
1/\widetilde{q}
=
(1-\theta)/2
+
\theta/q,
\quad
\widetilde{s} 
=
(1-\theta)t + \theta s,
\end{align*}
and
\begin{align*}
\widetilde{m}
&=
 - (N-1)\left(\frac{n}{2} - t\right) \times (1-\theta) 
+
\left(
\min
\left\{
\frac{n}{p}, \frac{n}{2}
\right\}
-
\sum_{j=1}^N
\max
\left\{
\frac{n}{p_j}, \frac{n}{2}
\right\}
+\sum_{j=1}^N s_j -s
\right)
\times 
\theta
\\
&=
\min
\left\{
\frac{n}{\widetilde{p}}, \frac{n}{2}
\right\}
-
\sum_{j=1}^N
\max
\left\{
\frac{n}{\widetilde{p}_j}, \frac{n}{2}
\right\}
+
\sum_{j=1}^N
\widetilde{s}_j
-\widetilde{s}.
\end{align*}

Since $1 < \widetilde{p}_j, \widetilde{q}_j, \widetilde{q} < \infty$ for sufficiently small $\theta$, it follows from
Lemma \ref{lemNecessity2} that 
\begin{align*}
&\widetilde{s}_j 
\le 
\max
\left\{
\frac{n}{\widetilde{p}_j},
\frac{n}{2}
\right\},
\quad
j=1, \dots, N,
\end{align*}
which means that
\begin{align*}
&s_j 
\le 
\max
\left\{
\frac{n}{p_j},
\frac{n}{2}
\right\}
+ 
\frac{1-\theta}{\theta} \left(\frac{n}{2} - t \right),
\quad
j=1, \dots, N,
\end{align*}
Taking the limit $t \to n/2$,  
we obtain the desired conclusion. 

On the other hand, since $1 < \widetilde{p}_j < \infty$ if $\theta$ is sufficiently small, Lemma \ref{lemNecessity1} yields that
\begin{align*}
&\widetilde{s} 
\ge 
-
\max
\left\{
\frac{n}{{\widetilde{p}}^{\prime}},
\frac{n}{2}
\right\}.
\end{align*}
We remark that $1< \widetilde{p} < \infty$ for  sufficiently small $\theta$,
and hence we can write the above inequality as
\[
s 
\ge 
-
\max
\left\{
\frac{n}{p^{\prime}},
\frac{n}{2}
\right\}
-
\frac{1-\theta}{\theta}
\left(
t
+
\frac{n}{2}
\right).
\]
We conclude that $s \ge -\max\{n/p^\prime, n/2\}$ by taking the limit $t \to -n/2$.

The proof of Theorem \ref{thmnec} is complete.
\end{proof}

\section*{Acknowledgement}
The author sincerely expresses his thanks to Professor Naohito Tomita 
for helpful advice and warm encouragement.

\end{document}